\documentclass[11pt,noinfoline]{imsart}  

\usepackage{amsmath}
\usepackage{amssymb}
\usepackage{amsthm}
\usepackage{amscd,enumerate}
\usepackage{amsfonts}
\usepackage{graphicx}%
\usepackage{fancyhdr, geometry}
\usepackage{url,  color, verbatim}
\usepackage[normalem]{ulem}
\usepackage[colorlinks]{hyperref} 
\usepackage{mathtools}
\usepackage{tikz}
\usetikzlibrary{matrix}
\usepackage{tabularx}
\usepackage{booktabs}

\numberwithin{equation}{section}

\theoremstyle{plain} 
\newtheorem{thm}{Theorem}[section]
\newtheorem{corollary}[thm]{Corollary}

\newtheorem{cor}[thm]{Corollary}
\newtheorem{conj}{Conjecture}
\newtheorem{lm}[thm]{Lemma}

\newtheorem{prop}[thm]{Proposition}
\theoremstyle{definition}

\newtheorem{example}[thm]{Example}

 \newcommand{\dc}[1]{\mathrm{reg}_0{(#1)}}

\newcommand{\umodel}{\mathcal{I}(n,D,p)}
\newcommand{\randomideal}{\mathfrak{I}}
\newcommand{\gmodel}{\mathcal{G}(n,D,\mathbf{p})}
\newcommand{\gensetdist}{\mathcal{B}(n,D,p)}
\newcommand{\randomgenset}{\mathfrak{B}}
\newcommand{\genset}{B}
\newcommand{\mingenset}{M}
\newcommand{\fixideal}{I}
\newcommand{\umodeltwovars}{\mathcal{I}(2,D,p)}

\newcommand{\var}[1]{\mathrm{Var}\left[#1\right]}
\newcommand{\prob}[1]{\mathrm{P}\left(#1\right)}
\newcommand{\expect}[1]{\mathbb{E}\left[#1\right]}
\newcommand{\cov}[1]{\mathrm{Cov}\left[#1\right]}
\newcommand{\littleoh}[1]{\mathrm{o}\left(#1\right)}
\newcommand{\littleomega}[1]{\mathrm{\omega}\left(#1\right)}
\newcommand{\bigoh}[1]{\mathrm{O}\left(#1\right)}
\newcommand{\ideal}[1]{\left(#1\right)}
\newcommand{\indic}{\mathbf 1}
\newcommand{\kring}{k[x_1,\ldots,x_n]}
\newcommand{\modelname}{ER-type model}
\newcommand{\intvecs}{\mathbb{Z}^n_{\geq 0}}
\newcommand{\initdeg}[1]{\mathrm{initdeg}(#1)}

\begin{document}
\title{Random Monomial Ideals\thanks{This work is supported by the  NSF collaborative grants DMS-1522662  to Illinois Institute of Technology and DMS-1522158 to the Univ. of California, Davis.}
}
\runtitle{Random monomial ideals}
\begin{center}
\author{Jes\'us A. De Loera, Sonja Petrovi\'c, Lily Silverstein, Despina Stasi, Dane Wilburne}
\runauthor{De Loera, Petrovi\'c, Silverstein, Stasi, Wilburne} 
\end{center}
\maketitle

{\bf Abstract:} Inspired by the study of random graphs and simplicial complexes, and motivated by the need to understand average behavior of ideals, we propose and study probabilistic models of random monomial ideals.  We prove theorems about the probability distributions, expectations and thresholds for events involving  monomial ideals with given Hilbert function, Krull dimension, first graded Betti numbers, and present several experimentally-backed conjectures about  regularity,  projective dimension, strong genericity, and Cohen-Macaulayness of random monomial ideals.
\medskip

\section{Introduction}

Randomness has long been used to study polynomials. 
A commutative algebraist's interest in randomness stems from the desire to understand ``average'' or ``typical'' behavior of ideals and rings.  A natural approach to the problem is to define a probability distribution on a set of ideals or varieties, which, in turn, induces a distribution on algebraic invariants or properties of interest.  In such a formal setup, questions of expected  (typical) or unlikely (rare, non-generic) behavior can be stated formally  using probability.  Let us consider some examples. Already in the 1930's, Littlewood and Offord \cite{littlewood+offord} studied the expected number of real roots of a random algebraic equation defined by random coefficients. The investigations on random varieties, defined by random coefficients on a fixed Newton polytope support, have generated a lot of work now considered classical (see, e.g., \cite{kac, Kouchnirenko1976, sturmfelssurvey} and the references therein).  A probabilistic analysis of algorithms called \emph{smooth analysis} has been used in algebraic geometry, see \cite{BePa08a,burgissercucker}. Roughly speaking, 
smooth analysis measures the expected performance of an algorithm under slight random perturbations of worst-case inputs.  Another more recent example of a notion of algebraic randomness appears in \cite{EinErmanLazarsfeld,EinLazarsfeld}, where they consider the distribution of Betti numbers generated randomly using concepts from the Boij-S\"oderberg  theory \cite{EisenbudShreyer}. 
Still, there are many other examples of algebraic topics in which probabilistic analysis plays a useful role (see e.g., \cite{ourviolatorpaper} and the references therein). Our paper introduces a new probabilistic model on monomial ideals inside a polynomial ring.  

Why work on a probabilistic model for monomial ideals? There are at least three good reasons: First, monomial ideals are the simplest of ideals and play a fundamental role in commutative algebra, most notably as Gr\"obner degenerations of general ideals, capturing the range of values shown in all ideals for several invariants (see \cite{coxlittleoshea,EisenbudCommAlgBook}). Second, monomial ideals provide a strong link to algebraic combinatorics (see \cite{herzog+hibi,miller+sturmfels,stanley}). Third,  monomial ideals naturally generalize graphs, hypergraphs, and simplicial complexes; and, since the seminal paper of Erd\"os and R\'enyi~\cite{Erdos+Renyi}, probabilistic methods have been successfully applied to study those objects (see, e.g., \ \cite{AlonSpencer, bollobasbook,  kahlesurvey} and the references therein). Thus, our work is an extension of both classical probabilistic combinatorics 
and the new trends in stochastic topology.

Our goal is to provide a formal probabilistic setup for constructing and understanding distributions of monomial ideals and the induced distributions 
on their invariants (degree, dimension, Hilbert function, regularity, etc.).  To this end, we define below a simple probability distribution on the set of monomial ideals.  Drawing a random monomial ideal from this distribution allows us to study the average behavior of algebraic invariants.  While we are interested in more general probability distributions on ideals, some of which we describe in Section \ref{sec:othermodels}, 
we begin with the most basic model one can consider to generate random monomial ideals; inspired by classical work, we call this family 
the \emph{Erd\H os-R\'enyi-type model},  or the  \emph{\modelname{} for random monomial ideals}. 

 \smallskip
\paragraph{\bf The \modelname{} for random monomial ideals} 
Let $k$ be a field and let $S=\kring$ be the polynomial ring  in $n$ indeterminates. 
We  construct a random monomial ideal in $S$ by producing a random set of generators as follows:  
Given an integer $D$ and a parameter $p=p(n,D)$, $0\leq p\leq 1$, we include {independently} \emph{with probability $p$} each non-constant monomial of total degree at most $D$ in $n$ variables  in a generating set of the ideal.   
In other words, starting with $\genset=\emptyset$, each monomial in $S$ of degree at most $D$ is added to the set $\genset$ independently with equal probability $p$. The resulting random monomial ideal is then simply $I=\ideal{\genset}$; if $\genset=\emptyset$, then we let $I=\ideal{0}$. 

Henceforth, we will denote  by $\gensetdist$ the resulting Erd\H os-R\'enyi-type distribution on the sets of monomials.  
Since sets of monomials are now random variables, a  random set of monomials $\randomgenset$ drawn from  
the distribution $\gensetdist$ will be denoted by the standard notation  for distributions $\randomgenset\sim\gensetdist$.  
Note that if $B\subset S$ is any fixed set of monomials of degree at most $D$ each  and $\randomgenset\sim\gensetdist$, 
then $$P(\randomgenset=B)=p^{|B|}(1-p)^{\binom{D+n}{D}-|B|-1}.$$

In turn, the distribution $\gensetdist$ induces a distribution $\umodel$ on ideals. 
We will use the notation $\randomideal\sim\umodel$ to indicate that $\randomgenset\sim\gensetdist$ 
and $\randomideal=\ideal{\randomgenset}$  is a random monomial ideal generated by the \modelname{}. 

\smallskip
Before we state our results, let us establish some necessary probabilistic notation and background. 
Readers already familiar with random structures and probability background may skip this part, but those interested in further details are directed to many excellent texts on random graphs and the probabilistic method including \cite{AlonSpencer, bollobasbook, JansonEtAl}.

Given a random variable $X$, we denote its \emph{expected value} by $\expect{X}$, its \emph{variance} by $\var{X}$, and its \emph{covariance} with random variable $Y$ as $\cov{X,Y}$.  We will often use in our proofs four well-known facts: the \emph{linearity} property  of expectation, i.e.,  $\expect{X+Y}=\expect{X}+\expect{Y}$, the \emph{first moment method} (a special case of Markov's inequality), which states that $\prob{X> 0}\le\expect{X}$, and the \emph{second moment method} (a special case of Chebyshev's inequality), which states that $\prob {X=0} \leq {\var{X}}/{(\expect{X})^2}$ for $X$ a non-negative integer-valued random variable. Finally, an \emph{indicator random variable} $\indic_A$ for an event $A$ is a random variable such that  $\indic_A=1$, if event $A$ occurs, and $\indic_A=0$ otherwise. Indicator random variables behave particularly nicely with regard to taking expectations: for any event $A$, $\expect{\indic_A}=P(A)$ and $\var{\indic_A}=P(A)(1-P(A))$. Following convention, we abbreviate by saying a property holds  \emph{a.a.s},
to mean that a property holds \emph{asymptotically almost surely}  if, over a sequence of sets, the probability of having the property converges to $1$. 

An important point is that $\umodel$ is a \emph{parametric} probability distributions on the set of  monomial ideals, because it clearly depends on the value of the probability parameter $p$.  A key concern of our paper is to investigate how invariants evolve as $p$ changes. To express some of our results, we will need to use asymptotic analysis of probabilistic events: For functions $f,g: \mathbb{N} \rightarrow \mathbb{R}$, we write $f(n)=\littleoh{g(n)}$, and $g(n)=\littleomega{f(n)}$ if \mbox{$\lim_{n\rightarrow\infty} f(n)/g(n)=0$.} We also write $f(n)\sim g(n)$ if $\lim_{n\rightarrow\infty} f(n)/g(n)=1$, and $f(n)\asymp g(n)$ if there exist positive constants $n_0, c_1,c_2$ such that $c_1g(n)\leq f(n) \leq c_2g(n)$ when $n\ge n_0.$
When a sequence of probabilistic events is given by $f(n)$ for $n\in\mathbb{N}$, we say that $f$ happens \emph{asymptotically almost surely}, abbreviated \emph{a.a.s.},  if \mbox{$\lim_{n\to\infty}\prob{f(n)}=1$.}

In analogy to graph-theoretic properties, we define an \emph{(monomial) ideal-theoretic property} $Q$ to be the set of all (monomial) ideals that have property $Q$. A property $Q$ is \emph{monotone increasing} if for any monomial ideals $I$ and $J$ such that $I\subseteq J$, $I\in Q$ implies $J\in Q$ as well. We will see that several algebraic invariants on monomial ideals are monotone. 

Let $\randomideal\sim\umodel$. A function $f:\mathbb{N}\to\mathbb{R}$ is a \emph{threshold function} for a monotone property $Q$ if for any $p:\mathbb{N}\to[0,1]$:
\[
\lim_{D\to\infty} \prob{\randomideal\in Q}=
\begin{cases}
0, & \text{ if } p= \littleoh{f(D)},\\
1, & \text{ if } p = \littleomega{f(D)}.
\end{cases}
\]
A similar definition holds for the case when $f$ is a function of $n$ and $n\to\infty$. A threshold function gives a zero/one law that specifies when a certain behavior appears.
\medskip

\paragraph{\bf Our results}  Our results, summarized in (A)-(E) below, describe the kind of monomial ideals generated by the  \modelname{}, 
 in the sense that we can get a handle on both the average  and the extreme behavior of these random ideals and how various ranges of the probability parameter $p$ control those properties. 
 The following can also serve as an outline of this paper. 

\smallskip 
\paragraph{\bf (A) Hilbert functions and the distribution of monomial ideals} 

In Section~\ref{sec:hilbertdistribution}, we begin our investigation of random monomial ideals  by showing that $\umodel$ is not a uniform 
distribution on all monomial ideals, but instead, the probability of choosing a particular monomial ideal under the \modelname{} is completely 
determined by its Hilbert function and the first total Betti number of its quotient ring.   

For an ideal $I$, denote by $\beta_{i,j}=\beta_{i,j}(S/I)$  the $(i, j)$-th graded Betti number of  $S/I$, that is, the  number  of  syzygies of degree $j$ at step $i$ of the minimal free resolution. We will  collect the \emph{first graded Betti numbers}, counting the minimal generators of $I$, in a vector $\hat{\beta_1}=(\beta_{1,1},\beta_{1,2}, \dots, \beta_{1,reg_0})$, and denote by $\beta_1=\beta_{1,1}+\beta_{1,2}+ \cdots+ \beta_{1,reg_0}$ the \emph{first total Betti number} of $S/I$.  We will denote by $h_I(\cdot)$, or simply $h_I$, the Hilbert function of the ideal $I$.  We can now state the following foundational result:

\begin{thm}
\label{thm:yuge}
 Let $\fixideal\subseteq S$ be a fixed monomial ideal generated in degree at most $D$. 
The probability that the random monomial ideal $\randomideal\sim\umodel$ 
equals $\fixideal$ is 
	\[
	P(\randomideal=\fixideal)=p^{\beta_1(S/\fixideal)}(1-p)^{-1+\sum_{d=1}^Dh_\fixideal(d)}.
	\]
\end{thm}

In this way, two monomial ideals with the same Hilbert function and the same number of minimal generators have the same probability of occurring. 
Then, the following natural question arises:  How many monomial ideals in $n$ variables with generators of degree less than or equal to $D$ and 
with a given Hilbert function are there? It is well-known that one can compute the Hilbert function from the graded Betti numbers (see e.g.,  \cite{stanley}). 
Thus it is no surprise that we can state a combinatorial lemma to count monomial ideals in terms of their graded Betti numbers. 
If we denote by $NMon(n,D,h,\hat{\beta_1})$ the number of possible monomial ideals in $n$ variables, generated in degree no more than $D$, 
Hilbert function $h$, and first graded Betti numbers $\hat{\beta_1}=(\beta_{1,1},\beta_{1,2}, \dots, \beta_{1,reg_0})$, then $NMon(n,D,h,\hat{\beta_1})$ 
is  equal to the number of $0$--$1$ vertices of a certain convex polytope (Lemma~\ref{monopoly}).

Based on Theorem \ref{thm:yuge} and Lemma \ref{monopoly}, we can provide a closed formula for the induced distribution on Hilbert functions 
under the \modelname{}, that is, the probability that the \modelname{} places on  any given Hilbert function. This is presented in Theorem \ref{thm:beautiful}.

\smallskip
\paragraph{\bf (B) The Krull dimension and random monomial ideals}

 Section \ref{sec:krull} is dedicated to a  first fundamental ring invariant: the Krull dimension $\dim S/\randomideal$. 
By encoding  the Krull dimension as the transversal number of a certain hypergraph, we can show (see  Theorem~\ref{thm:krulldim} 
for the precise statement) that the probability  that $\dim(S/\randomideal)$ is equal to $t$ for $0\leq t\leq n$, for a random monomial ideal 
$\randomideal\sim\umodel$, is given by a polynomial in $p$ of degree $\sum_{i=1}^{t+1} {D\choose i}{n\choose i}$. This formula is 
exponentially large  but we make it explicit in some interesting values of $t$ (Theorem ~\ref{thm:KrullDimSpecificValues}); which in 
particular give a complete description of the case for two- and three-variable polynomials. 

Turning to asymptotic behavior, we prove that the Krull dimension of $(S/\randomideal)$ can be controlled by bounding 
the asymptotic growth of the probability parameter $p=p(D)$ as $D\to\infty$. 
The evolution of the Krull dimension in terms of $p$ is illustrated in Figure~\ref{fig:rangesDtoinfty}. 
The result is obtained by combining the family of threshold results from Theorem~\ref{thm:KrullDimThreshold} 
and is stated in the following corollary:  

\begin{cor} \label{cor:KrullDimThresholdRanges}
Let  $\randomideal\sim\umodel$, $n$ be fixed, and $0\leq t<n$. If the  parameter $p=p(D)$ is 
such that $p=\littleomega{D^{-(t+1)}}$ and $p=\littleoh{D^{-t}}$  as $D\to\infty$, then 
 $\dim(S/\randomideal)=t$  asymptotically almost surely.
\end{cor}

It is very useful to consider the evolution --  as the probability $p$ increases from $0$ to $1$ -- of the random monomial ideal from the \modelname{} 
and its random generating set $\randomgenset$. For very small values of $p$, $\randomgenset$ is all but guaranteed to be empty, and the random 
monomial ideal  is asymptotically almost surely the zero ideal. As $p$ increases, the random monomial ideal  evolves into a more complex ideal 
generated in increasingly smaller degrees 
and support. Simultaneously, as the density of $\randomgenset$ continues to increase with $p$, smaller-degree generators appear; these divide increasingly larger numbers of monomials and the random ideal starts to become less complex as its minimal generators begin to have smaller and smaller support, causing  the Krull dimension to drop. 
 Finally, the random ideal becomes $0$-dimensional and continues to evolve towards the maximal ideal.
 
\begin{figure}[h]
\includegraphics[width=1\linewidth]{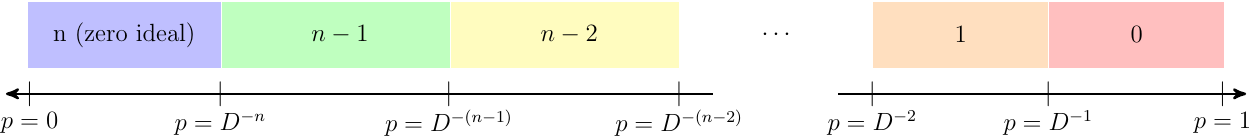}
\caption{Evolution of  the Krull dimension of $S/\randomideal$, where $\randomideal$ is the  random monomial ideal $\randomideal\sim\umodel$, as the probability parameter $p=p(D)$ changes.}   
\label{fig:rangesDtoinfty}
\end{figure}

\smallskip
\paragraph{\bf (C) First Betti numbers and random monomial ideals} 
The \modelname{} induces a distribution on the Betti numbers of the coordinate ring of random monomial ideals. 
To study this distribution, we first ask: \emph{ what are the first Betti numbers generated under the model?}  
A natural way to `understand' a distribution is to compute its expected value, i.e., the average first Betti number. 
In Theorem \ref{thm:mingennvars} we establish the asymptotic behavior for the expected number of minimal generators in a random monomial ideal.

We also establish and quantify threshold behavior of first graded Betti numbers of random monomial ideals $\randomideal\sim\umodel$
in two regimes, as $n$ or $D$ go to infinity, in Theorem~\ref{thm:gradedBettiThreshold}. As in the Krull dimension case, we combine the threshold results to obtain the following corollary:
\begin{cor} \label{cor:gradedBettiThresholdRanges}
Let  $\randomideal\sim\umodel$.
\begin{enumerate}[a)]
\item Let $D$ be fixed, and $d$ be a constant such that $1<d\leq D$. If the  parameter $p=p(n)$ is such that $p(n)=\littleomega{n^{-d}}$ and $p(n)=\littleoh{n^{-d+1}}$ then $\initdeg{\randomideal}=d$  asymptotically almost surely.
\item Let $n$ be fixed. Suppose that $d_i=d_i(D)$, $1 \leq d_i\leq D$ and \mbox{$\lim_{D\to\infty}d_i(D)=\infty$} for $i\in\{1,2\}$. If the  parameter $p=p(D)$ is such that $p(D)=\littleomega{{d_1}^{-n}}$ and $p(D)=\littleoh{{d_2}^{-n}}$ as $D\to\infty$, then  $d_2\leq \initdeg{\randomideal}\leq d_1$  asymptotically almost surely. Note that this is attainable when $d_1(D)=\littleomega {d_2(D)}$.
\end{enumerate}
\end{cor}

While Corollary~\ref{cor:gradedBettiThresholdRanges} explains how the parameter $p$ controls the smallest non-zero first Betti number, 
it gives no information regarding the largest such number; the latter, of course, leads the complexity of the first Betti numbers. 
To that end, we study the degree complexity $\dc{\randomideal}$, introduced by Bayer and Mumford in \cite{bayer+mumford} 
as the maximal degree of any reduced Gr\"obner basis of $I$;  for monomial ideals, this is simply the highest degree of a minimal generator. 
In Theorem~\ref{thm:degreecomplexity} we show that $p$ can be specified so that, even though $\genset$ contains many large degree monomials, none are minimal.

Intuitively, one can think of the above results in following terms:  when the growth of $p$ is bounded \emph{above} by that of $1/D$, 
Corollary~\ref{cor:gradedBettiThresholdRanges} provides the minimum degree of a minimal generator of $\randomideal$. When the growth of $p$ is bounded \emph{below} by that of $1/D$,  Theorem~\ref{thm:degreecomplexity}   provides the maximum degree of a minimal generator of $\randomideal$.  
The evolution of the minimum degree of a minimal generator for the case when $D$ is fixed and $n$ tends to infinity is illustrated in Figure~\ref{fig:rangesDtoinftyBetti}: 
\begin{figure}[h]
\includegraphics[width=1\linewidth]{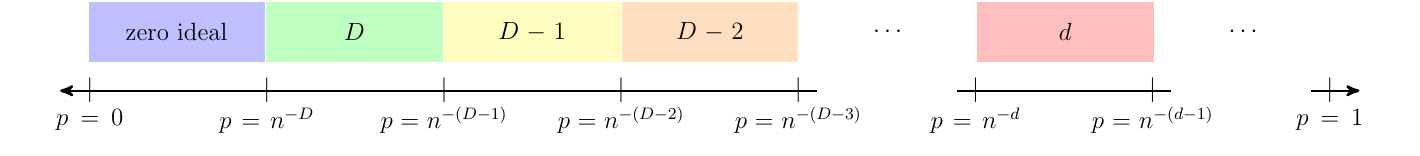}
\caption{Evolution of  the initial degree  of $S/\randomideal$, where $\randomideal$ is the  random monomial ideal $\randomideal\sim\umodel$, as the model parameter $p=p(n)$ changes, for fixed $D$. The initial degree, that is,  the smallest integer $k$ such that $\beta_{1,j}=0$  for $j< k$ and $\beta_{1,k}\neq 0$, is the value displayed in each box.}
\label{fig:rangesDtoinftyBetti}
\end{figure}

\smallskip
\paragraph{\bf (D) Other probabilistic models for generating monomial ideals} 

In Section \ref{sec:othermodels}  we define more general models for random monomial ideals. This is useful to study distributions on ideals that place varying probabilities on monomials depending on a property of interest. For example, one may wish to change the likelihood of monomials according to their total degree.
 To that end,  we define a \emph{general multiparameter model} for random monomials, show how the \modelname{} is a special case of it, and introduce the graded model. In Theorem~\ref{thm:CostaFarberCorresp} we show there is a choice of probability parameters in  the general model for random monomial ideals that recovers, as a special case, the 
 multiparameter model for random simplicial complexes of \cite{costafarber}, which itself 
 generalizes various models for random graphs and clique complexes presented in \cite{kahlesurvey,linial+meshulam}. 

\smallskip
\paragraph{\bf (E) Experiments \& conjectures} 

Section \ref{sec:experiments} contains a summary of computer simulations of additional algebraic properties of random monomial ideals using the \modelname{}. 
For each triple of selected values of model parameters $p$, $n$, and $D$, we generate $1000$  monomial ideals $\randomideal\sim\umodel$ and compute
their various properties. Our simulation results are summarized in Figures~\ref{fig:CM:grid}, \ref{fig:projdim:grid},  \ref{fig:gen:grid}, \ref{fig:regularity:grid}, \ref{fig:hom:radical}, and \ref{fig:hom:squarefree}.  The experiments
suggest several trends and conjectures, which we state explicitly.

\section{Hilbert functions and the Erd\H os-R\'enyi distribution on monomial ideals} \label{sec:hilbertdistribution}

The \emph{Hilbert function}, and the Hilbert polynomial that it determines, are important tools in the classification of ideals via the Hilbert scheme \cite{EisenbudCommAlgBook}.
Of course, monomial ideals are key to this enterprise: given an arbitrary  homogeneous ideal $I$ and a monomial order on $\kring$,  the monomial ideal generated by the leading terms of the polynomials in $I$ has the same Hilbert function as $I$. Therefore the study of 
all Hilbert functions reduces to the study of Hilbert functions of monomial ideals. Moreover, the {Hilbert function} $h_\fixideal(\cdot)$ of a monomial ideal $\fixideal\subseteq S$  has a useful combinatorial meaning: the value of $h_\fixideal(d)$, $d\le D$, counts the number of standard monomials - monomials of degree exactly $d$ that are 
\emph{not} contained in $\fixideal$. 

In this section we explore  the Erd\H os-R\'enyi distribution of monomial ideals and the induced distribution on Hilbert functions. It turns out they are intimately related. 

\subsection{Distribution of monomial ideals} 

Our first theorem precisely describes the distribution $\umodel$ 
on ideals specified by the \modelname{}.
Recall that Theorem~\ref{thm:yuge} states  that  the Hilbert function and first total Betti number determine the probability of any given ideal in $n$ variables generated in degree at most $D$.

\begin{proof}[Proof of Theorem \ref{thm:yuge}]  Fix $\fixideal\subseteq S$ generated in degree at most $D$ and let $\mingenset$ be the unique minimal set of generators of $\fixideal.$  Then, 
  $\randomideal=\fixideal$ if and only if $\randomgenset\supseteq \mingenset$ and no monomial $\mathbf{x}^{\alpha}$ such that $\mathbf{x}^{\alpha}\not\in \fixideal$ is in $\randomgenset$.  Let $A_1$ denote the event that each of the $\beta_1(S/\fixideal)$ elements of $\mingenset$ is in $\randomgenset$ and let $A_2$ denote the event that no monomial $\mathbf{x}^{\alpha}$ such that $\mathbf{x}^{\alpha}\not\in I$ is in $\randomgenset$.  Then, the event $\randomideal=\fixideal$ is equivalent to the event $A_1\cap A_2$.
Since the events $A_1$ and $A_2$ are independent, $P(\randomideal =\fixideal) = P(A_1\cap A_2) = P(A_1)P(A_2)$.
Observe that $P(A_1)=p^{\beta_1(S/\fixideal)}$, since each of the $\beta_1(S/\fixideal)$ elements of $\mingenset$ is chosen to be in $\randomgenset$ independently with probability $p$ and $P(A_2)=(1-p)^{-1+\sum_{d=1}^Dh_I(d)},$ since there are exactly $-1+\sum_{d=1}^Dh_I(d)$ monomials of degree at most $D$ not contained in $I$ and each of them is excluded from $\randomgenset$ independently with probability $1-p$.
\end{proof}

Consider a special case of this theorem in the following example.

\begin{example}[Principal random monomial ideals in 2 variables] Fix $n=2$ and $D>0$.  In this example, we calculate the probability of observing the principal random monomial ideal $\ideal{\randomgenset}=\ideal{ x^{\alpha}y^{\gamma}}\subseteq k[x,y]$, $\alpha+\gamma\le D$.  We observe the ideal $\ideal{ x^{\alpha}y^{\gamma}}$ exactly when $x^{\alpha}y^{\gamma}\in\randomgenset$ \emph{and} $x^{\alpha'}y^{\gamma'}\not\in\randomgenset$ for every monomial $x^{\alpha'}y^{\gamma'}$ with $\alpha'<\alpha\le D$ or $\gamma'<\gamma\le D$.  Thus, we must count the lattice points below the ``staircase" whose only corner is at $(\alpha,\gamma)\in\mathbb{Z}_{\ge0}^2$.  A simple counting argument shows that there are exactly $\frac{\gamma}{2}(2D-\gamma+3)+\frac{\alpha}{2}(2D-\alpha+3)-\alpha\gamma-1$ lattice points.  Hence,
$$\prob{\randomideal=\ideal{ x^{\alpha}y^{\gamma}}}=p(1-p)^{\frac{\gamma}{2}(2D-\gamma+3)+\frac{\alpha}{2}(2D-\alpha+3)-\alpha\gamma-1}.$$  Note that the right hand side of this expression is maximized at $\alpha=1, \gamma=0$ or $\alpha=0, \gamma=1$, so that among the principal ideals, the most likely to appear under the \modelname{} are $\ideal{x}$ and $\ideal{y}$.
\label{principal}
\end{example}

A corollary of Theorem \ref{thm:yuge} follows directly from the well-known formulas for  Stanley-Reisner rings:

\begin{corollary}
If $I$ is a square-free monomial ideal, then the probability of $I$  under the \modelname{} is determined by the number of minimal non-faces and faces of the associated simplicial complex.
\end{corollary}

\subsection{Distribution of Hilbert functions} 

Since the Erd\H os-R\'enyi-type model specifies a distribution on monomial ideals, it provides a formal probabilistic procedure to generate random Hilbert functions, 
in turn providing a very natural question: \emph{what is the induced probability of observing a particular Hilbert function? In other words, what is the most likely Hilbert function 
under the \modelname{}?  If we randomly generate a monomial ideal, are all Hilbert functions equally likely?} 

Not all possible non-negative vectors are first graded Betti vectors, but Lemma \ref{monopoly} below gives an algorithmic way to
determine, using polyhedral geometry, whether  a  particular potential set of first graded Betti  numbers can 
occur for a monomial ideal in $n$ variables. We note that Onn and Sturmfels introduced other polytopes useful in the study of 
initial (monomial) ideals of a zero-dimensional ideal for  $n$ generic points in affine $d$-dimensional space: the  \emph{staircase polytope} 
and the simpler \emph{corner cut polyhedron} \cite{onnsturmfels}. 

\begin{lm} \label{monopoly} 
Denote by $NMon(n,D,h)$ the number of possible monomial ideals in $n$ variables, with generating monomials of degree no more than $D$ and given Hilbert function $h$. 
Then $NMon(n,D,h)$ is equal to the number of vertices of the $0-1$ convex polytope $Q(n,D,h)$ defined by 

\begin{align*}
\sum_{|\alpha|=L} x_{\alpha} &= {n+L-1 \choose L} - h(L), \ \ \forall L=1,\dots D\\
& x_{\alpha} \leq x_{\gamma}, \quad \forall   \, \alpha \leq \gamma , \, |\alpha|+1 = |\gamma|, \qquad
\end{align*}

where  $\alpha, \gamma$ denote exponent vectors of monomials with $n$ variables and total degree no more than $D$, thus
the system has  ${n+D \choose D}-1$ variables.
\end{lm}

\begin{proof}[Proof of Lemma \ref{monopoly}]
The variables $x_{\alpha}$ are indicator variables recording when a monomial $x^{\alpha}$, for \mbox{$0<|\alpha|\leq D$,} is chosen to be in the ideal ($x_{\alpha} = 1$) or is not chosen to be in the ideal ($x_{\alpha} = 0$). The first type of equation forces the values of the Hilbert function at degree $L$ to be satisfied.
The  inequalities $x_{\alpha} \leq x_{\gamma}$, for all $\alpha\leq\gamma$, ensure the set of chosen monomials 
is closed under divisibility and thus forms an ideal. A non-redundant subset of these satisfying $|\alpha|+1=|\gamma|$ 
suffices to cut out the polyhedron. 
\end{proof}

\begin{thm} \label{thm:beautiful} Let $h$ 
 be a Hilbert function for an ideal of the polynomial ring in $n$ variables. As before,  $Q(n,D,h)$ is the polytope whose vertices are in one-to-one correspondence with the distinct monomial ideals in $n$ variables, with generators of degree less than or equal to $D$, and Hilbert function $h$.  Then the probability of generating a random monomial ideal $\randomideal$ from the \modelname{} with Hilbert function $h$ is expressed by the following formula

\begin{equation}\label{probabilityOfHilbert} 
P(h_\randomideal = h)= 
(1-p)^{-1+\sum_{d=1}^Dh(d)} \sum_{v \in V(Q(n,D,h))} p^{\beta_1(v)},
\end{equation}
where $\beta_1(v)$ is the total first Betti number of the monomial ideal associated to the vertex $v$.
\end{thm}

\begin{proof}[Proof of Theorem \ref{thm:beautiful}]
To derive the above formula for the probability of randomly generating a monomial ideal with a particular Hilbert function $h$ using the ER model, we decompose the random event into probability events that are combinatorially easy to count. 
The probability of generating random monomial ideals with a given Hilbert function $h$ is the sum of the probabilities of disjoint events enumerated by the vertices $v$ in the polytope $Q(n,D,h)$ (see Lemma \ref{monopoly}). 
For each choice of $v$, the probability of the corresponding ideal being randomly generated by our model is given in Theorem~\ref{thm:yuge} and is completely determined by its Hilbert function and first total Betti number. These ideals share a Hilbert function, $h$, and so $(1-p)^{-1+\sum_{d=1}^Dh(d)}$ is a common factor for all terms. The choice of $v$ determines the ideal and its first Betti number.
\end{proof}

Some final remarks about Theorem \ref{thm:beautiful}. First of all, one can find examples that show that the Hilbert function of an ideal does not uniquely determine the value of the number of minimal generators.
Second, the values for the Betti numbers have explicit tight bounds from the Bigatti-Hulett-Pardue theorem (see \cite{Bigatti},\cite{Hulett}, \cite{Pardue}). Finally, note that, from the point of view of computer science, it is useful to know that the Hilbert function $h$ can be specified with one polynomial (the Hilbert polynomial) accompanied by a list of finitely many values that do not fit the polynomial.

\section{Krull dimension} \label{sec:krull}

In this section, we study the Krull dimension $\dim S/\randomideal$ of a random monomial ideal under the model $\umodel$.
First, in Section~\ref{sec:krullPolynomial},  we give an explicit formula to calculate the probability of $\randomideal$ having a given Krull dimension.
Second, in Section~\ref{sec:krullThresholds}, we establish threshold results characterizing the asymptotic behavior of the Krull dimension and we provides ranges of values of $p$ that, asymptotically, control the Krull dimension of the random monomial ideal.

\subsection{Hypergraph Transversals and Krull dimension} 
\label{sec:Transversals}

Let us recall some necessary definitions. A \emph{hypergraph} $H$ is a pair $H=(V,E)$ such that $V$ is a finite vertex set and $E$ is a collection of non-empty subsets of $V$, called \emph{edges}; we allow edges of size one. A \emph{clutter} is a hypergraph where no edge is contained in any other (see, e.g., ~\cite[Chapter~1]{cornuejols}). 
A \emph{transversal} of $H$, also called a \emph{hitting set} or a \emph{vertex cover}, is a set $T\subseteq V$ such that no edge in $E$ has a nonempty intersection with $T$. The \emph{transversal number} of $H$, also known as its \emph{vertex cover number} and denoted $c(H)$, is the minimum cardinality of any transversal of $H$. 

Given a collection $\genset$ of monomials in $\kring$, the \emph{support hypergraph} of $\genset$ is the hypergraph $G(\genset)$ with vertex set $V=\{x_1,\ldots,x_n\}$, and edge set $E=\{\text{supp}({x}^\alpha)\mid {x}^\alpha\in \genset\}$, where $\textrm{supp}({x}^\alpha)$ denotes the support of the monomial ${x}^\alpha$, i.e., $x_i\in\textrm{supp}({x}^\alpha)$, if and only if $x_i\mid {x}^\alpha$. What we require here is 
that for $\fixideal=\ideal{\genset}$, the Krull dimension of $S/\fixideal$ is determined by the transversal number of $G(\genset)$ through the formula 
\begin{equation}\label{eq:KrullDim}
\dim S/\fixideal + c\left(G(\genset)\right) = n.
\end{equation}
This formula is known to hold for square-free monomial ideals minimally generated by $B$~(see e.g., \cite{Faridi05}). 
Since the support hypergraph of $\genset$ is equal to that of its maximal square-free subset and $c\left(G(\genset)\right)$ 
is invariant under the choice of generating set $\genset$ of $I$, the equation, obviously, also holds for general ideals. 

\subsection{Probability distribution of Krull dimension}
\label{sec:krullPolynomial}

The relationship between Krull dimension and hypergraph transversals leads to a complete characterization of the probability of producing a monomial ideal with a particular fixed Krull dimension in the \modelname{} (Theorem~\ref{thm:krulldim}). Explicitly computing the values given by Theorem~\ref{thm:krulldim} requires an exhaustive combinatorial enumeration, which we demonstrate for several special cases in Theorem~\ref{thm:KrullDimSpecificValues}. Nevertheless, where the size of the problem makes enumeration prohibitive, Theorem~\ref{thm:krulldim} gives that 
 $\prob{\dim S/\randomideal =t}$ is always a polynomial in $p$ of a specific degree, and thus can be approximated by
 numerically evaluating $\prob{\dim S/\randomideal =t}$ for a sufficient number of $p$ values,
and interpolating.

\begin{thm}\label{thm:krulldim} 
Let $\randomideal\sim\umodel$. For any integer $t$, $0\leq t\leq n$, the probability that $S/\randomideal$ has Krull dimension $t$ is given by a polynomial in $p$ of degree $\sum_{i=1}^{t+1} {D\choose i}{n\choose i}$. More precisely,
	\[
	\prob{\dim S/\randomideal =t} = \sum_{C\in \mathcal{C}_{n-t}}
	\prod_{\sigma\in E(C)} 1-(1-p)^{D\choose |\sigma|} 
\prod_{\substack{\sigma'\subset\{x_1,\ldots,x_n\}\\\sigma\not\subseteq \sigma'\forall \sigma\in E(C)}} (1-p)^{D\choose |\sigma'|}, 
	\] where $\mathcal{C}_{n-t}$ is the set of all clutters on $\{x_1,\ldots,x_n\}$ with transversal number $n-t$.
\end{thm}

\begin{proof}
Let $\randomgenset$ be the random generating set that gave rise to $\randomideal$. 
By Equation~\eqref{eq:KrullDim}, the probability that $S/\randomideal$ has Krull dimension $t$ is equal to the probability that $c(G(\randomgenset))=n-t$.
Let $G_{\min} (\randomgenset)$ be the
hypergraph obtained from $G(\randomgenset)$ by deleting all edges of $G(\randomgenset)$ that are strict supersets of other edges in $G(\randomgenset)$: 
in other words, $G_{\min}(\randomgenset)$ contains all edges $\sigma$ of $G(\randomgenset)$ for which there is no $\sigma'$ in $G(\randomgenset)$ such that $\sigma'\subsetneq \sigma. $

Also let $\mathcal{C}_{n-t}$ denote the set of clutters on $\{x_1,\ldots,x_n\}$ with transversal number $n-t$. Then $c(G(\randomgenset))=n-t$ 
if and only if $G_{\min}(\randomgenset)=C$ for some $C\in\mathcal{C}_{n-t}$. As these are disjoint events, we have
\begin{equation}
\prob{\dim S/\randomideal =t} = 
\sum_{C\in \mathcal{C}_{n-t}} \prob{G_{\min}(\randomgenset)= C}.
\label{eq:probdim1}
\end{equation}

Let  $\sigma\subset\{x_1,\ldots,x_n\}$. 
There are ${D\choose |\sigma|}$ monomials supported on $\sigma$, so
\[
\prob{\sigma\in E(G(\randomgenset))}=1-(1-p)^{D\choose |\sigma|},\text{ and }
\prob{\sigma\not\in E(G(\randomgenset))}=(1-p)^{D\choose |\sigma|}.
\]
Note that both expressions are polynomials in $p$ of degree exactly ${D\choose |\sigma|}$.
Now $G_{\min}(\randomgenset)= C$ if only if:
\begin{enumerate}
\item every edge $\sigma$ of $C$ is an edge of $G(\randomgenset)$, and
\item every $\sigma'\subset\{x_1,\ldots,x_n\}$ such that $\sigma'$ does not contain any $\sigma\in E(C)$  is \emph{not} an edge of $G(\randomgenset)$.\\
\end{enumerate} 
Note that condition 2 is the contrapositive of: any edge of $G(\randomgenset)$ which is not an edge of $C$ is not minimal.
Since edges are included in $G(\randomgenset)$ independently, the above are also independent events and the probability of their intersection is
\begin{align}\label{eq:probdim2}
\prob{G_{\min}(\randomgenset)= C}
&=\nonumber
\prod_{\sigma\in E(C)}
 \prob{\sigma\in E(G(\randomgenset))}
\prod_{\substack{\sigma'\subset\{x_1,\ldots,x_n\}\\\sigma\not\subseteq \sigma'\forall \sigma\in E(C)}} \prob{{\sigma'\not\in E(G(\randomgenset))}}\\
&=
\prod_{\substack{\sigma\in E(C)\\ \text{s.t.\ }|\sigma|\leq t+1}} \prob{\sigma\in E(G(\randomgenset))}
\prod_{\substack{\sigma'\subset\{x_1,\ldots,x_n\}\\ \text{s.t.\ }|\sigma'|\leq t \text{ and }\\
		\sigma\not\subseteq \sigma'\forall \sigma\in E(C)}} \prob{{\sigma'\not\in E(G(\randomgenset))}}.
\end{align}

The last equality follows because for $C\in \mathcal{C}_{n-t}$, if $\sigma$ is an edge of $C$ then $|\sigma|\leq t+1$, and also if $\sigma'$ is a subset of $\{x_1,\ldots,x_n\}$ satisfying $\sigma\not\subseteq \sigma'$ for all $\sigma\in E(C)$, then $|\sigma'|\leq t$. To show the first statement, suppose $\sigma\in E(C)$ with $|\sigma|>t+1$. Since $C$ is a clutter, no proper subset of $\sigma$ is an edge of $C$, so for every $\sigma'\in E(C)$, $\sigma'\neq \sigma$, $\sigma'$ contains at least one vertex not in $\sigma$. Hence the set $T=\{x_1,\ldots,x_n\}\backslash \sigma$ intersects every edge of $C$ except $\sigma$. By taking the union of $T$ with any one vertex in $\sigma$, we create a transversal of $C$ of cardinality $|T|+1=n-|\sigma|+1<n-t$, contradicting $c(C)=n-t$.

For the second statement, suppose $|\sigma'|>t$. By assumption no edge of $C$ is a subset of $\sigma'$, so every edge of $C$ contains at least one vertex in the set $T=\{x_1,\ldots,x_n\}\backslash \sigma'$. Hence $T$ is a transversal of $C$ with $|T|=n-|\sigma'|<n-t$, again a contradiction.

No subset of $\{x_1,\ldots,x_n\}$ can appear in both index sets of \eqref{eq:probdim2}, and each subset that does appear contributes a polynomial in $p$ of degree $D\choose i$. It follows that \eqref{eq:probdim2} is a polynomial in $p$ of degree no greater than $\sum_{i=1}^{t+1}{D\choose i}{n\choose i}$, and hence so is $\prob{\dim S/\randomideal =t}$ as \eqref{eq:probdim1} is a sum of such polynomials.

To prove this bound is in fact the precise degree of the polynomial, we show there is a particular clutter for which \eqref{eq:probdim1} has degree exactly $\sum_{i=1}^{t+1}{D\choose i}{n\choose i}$, and furthermore, that for every other clutter the expression in \eqref{eq:probdim2} is of strictly lower degree. Hence the sum over clutters in \eqref{eq:probdim1} has no cancellation of this leading term.

Consider the hypergraph $K^{t+1}_n$ that contains all ${n\choose {t+1}}$ edges of cardinality $t+1$ and no other edges. Then $K^{t+1}_n\in \mathcal{C}_{n-t}$ and
\begin{align*}
\prob{G_{\min}(\randomgenset)= K^{t+1}_n}
&=
\prod_{\substack{\sigma\in \{x_1,\ldots,x_n\}\\|\sigma|=t+1}} \prob{\sigma\in E(G(\randomgenset))}
\prod_{\substack{\sigma\in \{x_1,\ldots,x_n\}\\|\sigma|\leq t}} \prob{\sigma\not\in E(G(\randomgenset))}\\
&
=
(1-(1-p)^{D\choose{t+1}})^{n\choose{t+1}}(1-p)^{Dn+{D\choose 2}{n\choose 2}+\cdots+{D\choose t}{n\choose t}}
\end{align*}
which has the correct degree. 
On the other hand, if $C\in\mathcal{C}_{n-t}$ and $C\neq K^{t+1}_n$, then at least one edge $\sigma$ of $C$ is not an edge of $K^{t+1}_n$; hence $|\sigma|\leq t$. All subsets properly containing $\sigma$ are neither edges of $C$, nor do they satisfy condition 2 above, hence these subsets are not indexed by either product in \ref{eq:probdim2}. In particular there are positively many subsets of cardinality $t+1$ which do not contribute factors to $\prob{G_{\min}(\randomgenset)= C}$.
\end{proof}

To use the main formula of Theorem \ref{thm:krulldim}, one must be able to enumerate all clutters on $n$ vertices with transversal number $n-t$. 
When $t$ is very small or very close to $n$ this is tractable, as we see in the following theorem.

\begin{thm} \label{thm:KrullDimSpecificValues}
Let $\randomideal\sim\umodel$. 
Then,  
\begin{enumerate}[a)]

\item $\prob{\dim S/\randomideal=0}=\left(1-(1-p)^D\right)^n.$

\item $\prob{\dim S/\randomideal=1}=
\sum_{j=0}^{n-1}{n\choose j}(1-(1-p)^D)^j(1-p)^{D(n-j)}\left(1-(1-p)^{D\choose 2}\right)^{{n-j}\choose 2}.$
\item $\prob{\dim S/\randomideal=n-1}=
-(1-p)^{{{n+D}\choose n}-1}+\sum_{j=1}^{n} (-1)^{j-1}{n\choose j}(1-p)^{{{n+D}\choose n}-1-{{n+D-j}\choose n}}.$
\item $\prob{\dim S/\randomideal=n}=(1-p)^{{{n+D}\choose n}-1}.$
\end{enumerate}
For $n\leq 4$, the listed formulas gives the complete probability distribution of Krull dimension induced by $\umodel$, for any integer $D$.
\end{thm}

\begin{proof}
 \emph{Proof of Part (a)}
For $\randomgenset\sim\gensetdist$, $\randomideal=\ideal{ \randomgenset}$, $S/\randomideal$ is zero dimensional if and only if $G_{\min}(\randomgenset)\in\mathcal{C}_n$.
There is a single clutter on $n$ vertices with transversal number $n$: the one with edge set $\{\{x_1\},\{x_2\},\ldots,\{x_n\}\}$. 
Hence by Theorem~\ref{thm:krulldim}, 
\[
\prob{\dim S/\randomideal =0} = 
\prod_{i=1}^n 1-(1-p)^{D\choose {|\{x_i\}|}}
=
\left(1-(1-p)^D\right)^n.
\]

\noindent \emph{Proof of Part (b):}  $S/\randomideal$ is one-dimensional if and only if
 $G_{\min}(\randomgenset)\in\mathcal{C}_{n-1}$. We wish to describe $\mathcal{C}_{n-1}$. Suppose $C$ is a clutter on $n$ vertices and exactly $j$ of the vertices are contained in a 1-edge. Then $j\neq n$ else $C$ would be the clutter from part $a$, so let
 \mbox{$0\leq j\leq n-1$,} and denote by $V$ the set of these $j$ vertices. Then $V$ is a subset of any transversal of $C$. Let \mbox{$W= \{x_1,\ldots,x_n\}\backslash V$,} then it can be shown that
$c(C)=n-1$ if and only $E(C)= \{\{x_i\}\mid x_i\in V\}\cup\{\{x_i,x_k\}\mid x_i,x_k\in W, x_i\neq x_k\}$. 
Hence $\prob{G_{\min}(\randomgenset)=C}$ equals
\begin{align*}
\prod_{x_i\in V}\prob{\{x_i\}\in E(G(\randomgenset))}\prod_{x_i,x_k\in W}\prob{\{x_i,x_k\}\in E(G(\randomgenset))}\prod_{x_i\in W}\prob{\{x_i\}\not\in E(G(\randomgenset))}\\
 =
(1-(1-p)^D)^j\left(1-(1-p)^{D\choose 2}\right)^{{n-j}\choose 2}(1-p)^{D(n-j)}.
\end{align*}
The expression for $\prob{\dim S/\randomideal=1}$ is obtained by summing over all ${n\choose j}$ ways of selecting the $j$ 1-edges, for each $0\leq j\leq n-1$.

\noindent \emph{Proof of Part (c):}
For the case of $(n-1)$--dimensionality, Theorem \ref{thm:krulldim} requires us to consider clutters with transversal number 1: clutters where some vertex appears in all the edges. However, for this case we can give a simpler argument by looking at the monomials in $\randomgenset$ directly. Now the condition equivalent to $(n-1)$--dimensionality is that there is some $x_i$ that divides every monomial in $\randomgenset$.                                                      

Fix an $i$, then there are ${{{n+D}\choose{D}}-1-{{n+D-1}\choose{D-1}}}$ monomials that $x_i$ does \emph{not} divide. If $F_i$ is the event that $x_i$ divides every monomial in $\randomgenset$, then
$
\prob{F_i}=(1-p)^
{{{n+D}\choose{D}}-1-{{n+D-1}\choose{D-1}}}.
$
To get an expression for $(n-1)$-dimensionality, we need to take the union over all $F_i$, which we can do using an inclusion-exclusion formula considering the events that two variables divide every monomial, three variables divide every monomial, etc. Finally, we subtract the probability that $\randomgenset$ is empty.

\noindent \emph{Proof of Part (d):}
Since only the zero ideal has Krull dimension $n$, this occurs if and only if $\randomgenset$ is empty, which has probability $(1-p)^{{{n+D}\choose{D}}-1}$.
\end{proof}

\subsection{Threshold and asymptotic results} 
\label{sec:krullThresholds}

Next we study the asymptotic behavior of $\dim(S/\randomideal)$. In particular, we derive ranges of values of $p$ that control the Krull dimension as illustrated in Figure~\ref{fig:rangesDtoinfty}.  
As a special case we obtain a threshold result for the random monomial ideal being zero-dimensional. 

 
\begin{lm} 
\label{lm:supp}
Let $n, t$ be fixed integers, with $0\leq t\leq n-1$. If $\randomgenset\sim\gensetdist$ and $p(D)=\littleoh{D^{-t}}$, then a.a.s. $\randomgenset$ will contain no monomials of support size 
 $t$ or less as $D$ tends to infinity.
\end{lm}

\begin{proof}  
By the first moment method, the probability that $\randomgenset$ contains some monomial of support at most $t$ is bounded above by the expected number of such monomials. As the number of monomials in $n$ variables with support of size at most $t$ is strictly less than $\binom{n}{t}\binom{D+t}{t}$, the expectation is bounded above by the quantity $p\binom{n}{t}\binom{D+t}{t}$.  This quantity tends to zero when $p(D)=\littleoh{D^{-t}}$ and $n$ and $t$ are constants, thus establishing the lemma.
\end{proof}

\begin{thm} \label{thm:KrullDimThreshold}
Let $n$, $t$ be integers and $p=p(D)$ with $1\leq t\leq n$ and $0\leq p\leq 1$. Suppose that \mbox{$\randomgenset\sim\gensetdist$} and $\randomideal\sim\umodel$. Then $D^{-t}$ is a threshold function for the property that \mbox{$\dim(S/\randomideal) \leq t-1$.} 
In other words,
\[ \lim_{D\to\infty} \prob{\dim(S/\randomideal) \leq t-1}=
\begin{cases}
 0,& \text{ if }p=\littleoh{D^{-t}}\\
 1,& \text{ if }p=\littleomega{D^{-t}}
 \end{cases}.
 \]
\end{thm}

\begin{proof}  Let $p=\littleoh{D^{-t}}$. Consider the support hypergraph of $\randomgenset$, $G(\randomgenset)$. If every monomial in $\randomgenset$ has support of size $t+1$ or more, then  any $(n-t)$-vertex set of $G(\randomgenset)$ is a transversal, giving $c(G(\randomgenset))\leq n-t$. Thus Equation~\ref{eq:KrullDim} and Lemma~\ref{lm:supp} give that $\dim(S/\randomideal)\leq t-1$ holds with probability tending to 0.
 
Now, let $p=\littleomega{D^{-t}}$. 
Consider $\sigma\subset\{x_1,\ldots,x_n\}$ with $|\sigma|=t$  and let $X_{\sigma}$ be the random variable that records the number of monomials in $\randomgenset$ with support exactly the set $\sigma$. There are exactly $\binom{D}{t}$ such monomials of degree at most $D$: each monomial can be divided by all variables in $\sigma$ and the remainder is a monomial in (some or all of) the variables of $\sigma$ of degree no greater than $D-t$. 

Thus, $\expect{X_{\sigma}}=\binom{D}{t}p$ tends to infinity with $D$, when $p=\littleomega{D^{-t}}$.  Further, $\var{X_{\sigma}}\le\expect{X_{\sigma}}$, as $X_{\sigma}$ is a sum of independent indicator random variables, 
one for each monomial with support exactly $\sigma$. 
Hence we can apply the second moment method to conclude that as $D$ tends to infinity, 
\[\prob{X_{\sigma}=0}\le\frac{\var{X_{\sigma}}}{\expect{X_{\sigma}}^2}\le\frac{1}{\binom{D}{t}p }\to 0.\]  
In other words, a.a.s.\ $\randomgenset$ will contain a monomial with support $\sigma$.  Since $\sigma$ 
was an arbitrary $t$-subset of the variables $\{x_1,\ldots,x_n\}$, it follows that the same holds for every choice of $\sigma$.

The probability that any of the events ${X_{\sigma}=0}$ occurs is bounded above by the sum of their probabilities and tends to zero, as there are a finite number of sets $\sigma$. Then, with probability tending to 1, $G(\randomgenset)$ contains all size-$t$ edges, which implies that, for every vertex set 
of size $n-t$ or less,  $G(\randomgenset)$ contains an edge of size $t$ that is disjoint from it. Thus 
the transversal number of the hypergraph is at least $n-t+1$, 
and, by Formula~\ref{eq:KrullDim}, $\dim(S/\randomideal)=n-c(G(\randomgenset))\leq t-1$ a.a.s.

\end{proof}

The smallest choice of $t$-values in Theorem~\ref{thm:KrullDimThreshold},  $t=1$, 
gives a threshold function for the event that $\randomideal\sim\umodel$ has zero Krull dimension: 
\begin{cor} Let $\randomideal\sim\umodel$, $n$ be fixed and $p=p(D)$. Then the function ${1}/{D}$ is a threshold function for 
the $0$-dimensionality of $S/\randomideal$.
\label{thm:zerodim}
\end{cor}
 
 Now we move to obtain Corollary~\ref{cor:KrullDimThresholdRanges} announced in the Introduction. Theorem~\ref{thm:KrullDimThreshold} establishes a threshold result for each choice of constant $t$. But, if both $p=\littleomega{D^{-(t+1)}}$ and $p=\littleoh{D^{-t}}$ hold, then the theorem gives that events $\dim(S/\randomideal) > t$ and  $\dim(S/\randomideal) < t$ each hold with probability tending to $0$. Therefore the probability of their union, in other words the probability that $\dim(S/\randomideal)$ is either strictly larger than $t$ or strictly smaller than $t$, also tends to $0$.
By combining the threshold results applied to two consecutive constants $t$ and $t+1$ in this way, we have established the Corollary.

\section{Minimal generators of random monomial ideals from the \modelname{}} \label{sec:bettistuff}

A way to measure the complexity of a monomial ideal is to study the number of minimal generators. 
As before,  we will denote by $\beta_1:=\beta_1(S/\fixideal)$  the first total Betti number of the quotient 
ring of the ideal $I$, that is, the total number of minimal generators of $I$. 
Subsection~\ref{sec:bettiTotal} addresses this invariant and provides a formula for the asymptotic average value of $\beta_1$. 
A more refined invariant is, of course, the set of first \emph{graded Betti numbers} $\beta_{1,d}(S/\fixideal)$, recording the number of 
minimal generators of $\fixideal$ of degree $d$.  Subsection~\ref{sec:bettiGraded} provides threshold results that explain how 
the asymptotic growth of the parameter value $p$ influences the appearance of minimal generators of given degree. 
Subsection~\ref{sec:degComplexity} provides an analogous result for the \emph{degree complexity} $\dc{I}$ of a monomial ideal $I$, 
that is, the maximum degree of a minimal generator. 

\subsection{The first total Betti number}\label{sec:bettiTotal}

Consider the random  variable $\beta_1(S/\randomideal)$  for $\randomideal\sim\umodel$ in the regime $D\to\infty$.  
Before stating the results of this section, we introduce a concept from multiplicative number theory.  
The \emph{order $r$ divisor function}, denoted by $\tau_r(k)$, is a function that records the number of ordered factorizations of an integer $k$ into 
exactly $r$ parts (e.g.,\ $\tau_3(4)=6$). 

\begin{thm}\label{thm:mingennvars}
Let $\randomideal\sim\umodel$ and let $\beta_1(S/\randomideal)$ be the random variable denoting the number of minimal generators of  the 
random monomial ideal $\randomideal$.  Then, 
\begin{equation}
\lim_{D\to\infty}\expect{\beta_1(S/\randomideal)}= p\sum_{k=0}^{\infty}\tau_n(k+2)(1-p)^k.
\label{eq:mingens}\tag{3}
\end{equation}
\end{thm}

\begin{proof} 
For each $\alpha=(\alpha_1,\dots,\alpha_n)\in\intvecs$, let $\indic_{\alpha}$ be the indicator random variable for the event that $x^{\alpha}$ is a minimal generator of $\randomideal$. Excluding the constant monomial $1$ and $x^{\alpha}$ itself, there are $\prod_{i=1}^n(\alpha_i+1)-2$ monomials in $S$ that divide $x^{\alpha}.$
Therefore, $\expect{\indic_{\alpha}}=p(1-p)^{\prod_{i=1}^n(\alpha_i+1)-2}$ and letting $D\to\infty$ shows that  
\[
\lim_{D\to\infty}\expect{\beta_1(S/\randomideal)}=\sum_{0\not=\alpha\in\intvecs}\expect{\indic_{\alpha}}\ =\ p\sum_{0\not=\alpha\in\intvecs}(1-p)^{\prod_{i=1}^n(\alpha_i+1)-2}.
\]
Since $\prod_{i=1}^n(\alpha_i+1)-2$ is also the number of lattice points contained in the box $\prod_{i=1}^n[0,\alpha_i]$, excluding the origin and $\alpha$,  it follows that every integer $k\ge0$ appears as an exponent of $q$ in the infinite sum above exactly as many times as $k+2$ can be factored into a product of $n$ ordered factors, so that the above sum may be rewritten as $\sum_{k=0}^{\infty}\tau_n(k+2)(1-p)^{k}$.
\end{proof}

When $n=2,$ $\tau_2(k)$ is simply the usual divisor function which counts the number of divisors of an integer $k$.  The ordinary generating function of $\tau_2(k)$ has the form of a \emph{Lambert series} (see \cite{dilcher}), which leads to the following result in the two-variable case.

\begin{cor}\label{cor:mingen2vars}Let $\randomideal\sim\umodeltwovars$.  Then, 
\[
\lim_{D\to\infty}\expect{\beta_1(S/\randomideal)}=\frac{p}{(1-p)^2}\left[\sum_{k=1}^{\infty}\frac{(1-p)^{k}}{1-(1-p)^k}\right]-\frac{p}{1-p}.
\]

\end{cor}

\begin{proof} This follows from the Lambert series identity
$\sum_{k=1}^{\infty}\frac{(1-p)^k}{1-(1-p)^k}=\sum_{k=1}^{\infty}\tau_2(k)(1-p)^k$ in \eqref{eq:mingens}.
\end{proof}

Corollary~\ref{cor:mingen2vars} provides an expression for the asymptotic expected number of minimal generators of a random monomial ideal in two variables that can be easily evaluated for any fixed value of $p$.  For example, if $\randomideal\sim\mathcal{I}(2,D,10^{-5}),$ then asymptotically (i.e., 
for very large values of $D$) one expects $\randomideal$ to have about 12 generators on average.  On the other hand, when $n\ge3$, the right-hand 
side of \eqref{eq:mingens} is difficult to compute exactly, but one can obtain bounds from the fact that $n\le\tau_n(k)\le k^n,$ which are valid for every $k\ge2$. Hence
\begin{equation}
n\le\lim_{D\to\infty}\expect{\beta_1(S/\randomideal)}\le\frac{p}{(1-p)^2}\mathrm{Li}_{-n}(1-p),
\label{eq:mingensest}\tag{4}
\end{equation} where $\mathrm{Li}_{s}(x)$ denotes the \emph{polylogarithm function of order} $s$ defined by $\mathrm{Li}_s(x)=\sum_{k=1}^{\infty}\frac{x^k}{k^s}$.  A polylogarithm of negative integer order can be expressed as a rational function of its argument: 
\[
\mathrm{Li}_{-n}(1-p)=\sum_{j=0}^nj!S(n+1,j+1)\left(\frac{1-p}{p}\right)^{j+1},
\] where $S(n+1,k+1)$ denotes a \emph{Stirling number of the second kind} (see \cite{lewin}).  The upper bound in \eqref{eq:mingensest} is thus of order $\bigoh{p^{-n}}.$

\subsection{The first graded Betti numbers}\label{sec:bettiGraded}
Here we study the  behavior of the random  variables $\beta_{1,k}(S/\randomideal)$  for $\randomideal\sim\umodel$. Specifically, we show how the choice of $p$ in the \modelname{} controls the minimum degree of a generator of $\randomideal$, by controlling the degrees of the monomials in the random set $\randomgenset$.

\begin{lm}\label{lm:size.Bd.threshold}
Let $\randomgenset\sim \gensetdist$, $\randomideal=\langle\randomgenset\rangle$ and 
$
\randomgenset_d := \{x^{\alpha}\in \randomgenset: |\alpha|\le d\}.
$
\begin{enumerate}[a)]
\item If $d,D$ is fixed, $0<d\leq D$ and $p=p(n)$, then $f(n)=n^{-d}$ is a threshold function for the property that $|\randomgenset_d|=0.$

\item If $n$ is fixed, $d=d(D)\leq D$ is such that $\lim_{D\to\infty} d(D)\to\infty$ and $p=p(D)$, then   $f(d)=d^{-n}$ is a threshold function for the property that $|\randomgenset_d|=0.$
\end{enumerate}
\end{lm}

\begin{proof}
For part a),  we must show that when $p(n)=\littleoh{n^{-d}}$, then $|\randomgenset_d|=0$ 
with probability tending to 1, and when $p(n)=\littleomega{n^{-D}}$, then $|\randomgenset_d|>0$ with probability tending to 1. 

For $d$ fixed, the number of monomials of degree at most $d$ in $n$ variables is $\binom{n+d}{d}\asymp n^{d}$ (by considering the usual asymptotic bounds for binomial coefficients). 
For any $\alpha\in{\mathbb Z}^n_{\geq0}$, with $0<||\alpha||_1\leq d$, let $\indic_{\alpha}$ be the indicator random variable for the event that $x^\alpha\in \randomgenset$, so that
$
	|\randomgenset_d|=\sum_\alpha
	\indic_{\alpha},
$ and therefore 
$
	\expect{|\randomgenset_d|} = \sum_{\alpha}\expect{\indic_{\alpha}}=p\cdot\left({{n+d}\choose{d}}-1\right).
$
By the first moment method:
\[ 
	\prob{|\randomgenset_d|>0} \leq \expect{|\randomgenset_d|} = p\cdot\left({{n+d}\choose{d}}-1\right)\leq p e^dn^d.
\]
If $p(n)=\littleoh{n^{-d}}$, then $p e^dn^d\to 0$ as $n\to \infty$ and hence $\lim_{n\rightarrow\infty} \prob{|\randomgenset_d|>0}=0$. 

Now suppose $p(n)=\littleomega{n^{-d}}$.  In this case $\expect{|\randomgenset_d|}\rightarrow\infty$, since:
\[
	\expect{|\randomgenset_d|} =p\cdot\left({{n+d}\choose{d}}-1\right)\geq\left(\frac{2}{d}\right)^d\cdot p n^d.
\] The variance of $|\randomgenset_d|$ is calculated as follows: 
\[ 
	\var{|\randomgenset_d|} = \sum_{\alpha} \var{\indic_{\alpha}}+\sum_{\alpha\neq\alpha'} \cov{\indic_{\alpha},\indic_{\alpha'}}  = 
		\left( {{n+d}\choose{d}}-1 \right) \cdot p \cdot (1-p)=\expect{|\randomgenset_d|}\cdot(1-p),	 
\] where we have again used that $\var{\indic_{\alpha}} = \expect{\indic_{\alpha}^2}-\expect{\indic_\alpha}^2 = p-p^2$ 
and for $\alpha\not=\alpha'$, $\indic_{\alpha}$ and $\indic_{\alpha'}$ are independent and thus have $0$ covariance. 
Then by the second moment method and the fact that $\expect{|\randomgenset_d|}\rightarrow\infty$ when \mbox{$p=\littleomega{n^{-d}}$:}
\[\prob{|\randomgenset_d|=0}\leq \frac{\var{|\randomgenset_d|}}{(\expect{|\randomgenset_d|})^2} = \frac{1-p}{\expect{|\randomgenset_d|}}\leq \frac{1}{\expect{|\randomgenset_d|}}\to0\text{ as } n\to\infty.\]

Part b) follows in a similar fashion, once one notices that, for fixed $n$, ${{n+d}\choose{d}}={{n+d}\choose{n}}\asymp d^{n}$ as $d$ tends to infinity. 
\end{proof}

From Lemma~\ref{lm:size.Bd.threshold} we immediately obtain a threshold result for the first graded Betti numbers: 

\begin{thm}\label{thm:gradedBettiThreshold} 
Let  $\randomgenset\sim \gensetdist$ and  $\randomideal=\ideal{ \randomgenset}$. 
\begin{enumerate}[a)]
\item If $d\le D$ is fixed and $p=p(n)$, then the function $f(n)=n^{-d}$ is a threshold function for the property that 
$\initdeg{\randomideal}>d$.
\item If $n$ is fixed, $d=d(D)\leq D$ is such that $\lim_{D\to\infty}d(D)=\infty$ and $p=p(D)$, then $f(d)=d^{-n}$
 is a threshold function for the property that 
 $\initdeg{\randomideal}>d$. 
\end{enumerate}
\end{thm}

\begin{proof}  By definition, $\initdeg{\randomideal}>d$ if and only if $\sum_{k=1}^d\beta_{1,k}(S/\randomideal)=0.$  Equivalently, $\initdeg{\randomideal}>d$ if and only if $\randomgenset\not=\emptyset$ and $\randomgenset$ contains no monomials of degree $k$ for each $1\le k\le d$ and therefore
\[
P\left(\sum_{k=1}^d\beta_{1,k}(S/\randomideal)=0\right)=P\left(|\randomgenset_d|=0\right).
\]   Both a) and b) now follow directly from Lemma~\ref{lm:size.Bd.threshold}.
\end{proof}


The case $d=D$ of Theorem~\ref{thm:gradedBettiThreshold} deserves special mention, as it provides a threshold function for the property that the random ideal  generated by the \modelname{} is the zero ideal.

\begin{cor} Let $\randomgenset\sim \gensetdist$ and $\randomideal=\langle\randomgenset\rangle$.
\begin{enumerate}[a)]
\item Let $D$ be fixed and $p=p(n)$. Then $n^{-D}$ is a threshold function for the property that $\randomideal=\ideal{0}.$

\item Let $n$ be fixed and $p=p(D)$. Then $D^{-n}$ is a threshold function for the property that $\randomideal=\ideal{0}.$
\end{enumerate}
\label{cor:zeroideal}
\end{cor}
On the other hand, in the critical region of this threshold function, both the expected number of monomials in $\randomgenset$ and the variance of the number of monomials in $\randomgenset$ is constant.

\begin{prop}\label{prop:zero.ideal.critical.region}
Let $\randomgenset\sim \gensetdist$. 
\begin{enumerate}[a)]
\item For any $n, D$, if $p=1/\left({{n+D}\choose{D}}-1\right)$ then $\expect{|\randomgenset|}=1$ and $\var{|\randomgenset|}=1$.
\item For $D$ fixed, if $p=p(n)\asymp n^{-D}$ 
then {$k_1\leq \lim_{n\rightarrow\infty}\expect{|\randomgenset|}\leq k_2$} and {$c_1\leq\lim_{n\rightarrow\infty}\var{|\randomgenset|}\leq c_2$,} for some constants $k_1,k_2$, $c_1,c_2$.
\item For $n$ fixed, if $p=p(D)\asymp D^{-n}$ 
then {$k_1\leq \lim_{D\rightarrow\infty}\expect{|\randomgenset|}\leq k_2$} and {$c_1\leq\lim_{D\rightarrow\infty}\var{|\randomgenset|}\leq c_2$,} for some constants $k_1,k_2$, $c_1,c_2$. \end{enumerate}
\end{prop}
\begin{proof}
Since $\randomgenset=\randomgenset_D$, each of these claims follows directly from the calculated values of $\expect{|\randomgenset_D|}$ and $\var{|\randomgenset_D|}$ from the proof of Lemma~\ref{lm:size.Bd.threshold}.
\end{proof}

\smallskip

By combining threshold functions from Theorem~\ref{thm:gradedBettiThreshold} we can obtain Corollary~\ref{cor:gradedBettiThresholdRanges} announced in the Introduction: the result specifies ranges of $p$ for which $\initdeg{\randomideal}=d$ asymptotically almost surely. The argument is identical to the one used in Section~\ref{sec:krullThresholds} to obtain Corollary~\ref{cor:KrullDimThresholdRanges} from Theorem~\ref{thm:KrullDimThreshold}.

\subsection{Degree complexity}\label{sec:degComplexity}

In the previous subsection, we saw how the choice of $p$ influences the smallest degree of a minimal generator of a random monomial ideal.  In this subsection, we ask the complementary question: what can one say about the \emph{largest} degree of a minimal generator of a random monomial ideal?  The following theorem establishes an asymptotic bound for the degree complexity $\dc{\randomideal}$ of $\randomideal\sim\umodel$ for certain choices of probability parameter $p$.

\begin{thm} \label{thm:degreecomplexity}
Let $n$ be fixed, $\randomideal\sim\umodel$, $p=p(D)$ and let $r=r(D)$ be a function tending to infinity as $D\to\infty$.  If $p=\littleomega{\frac{1}{r}}$, then $\dc{\randomideal}\le nr$ a.a.s.
\end{thm}

\begin{proof}  
By the proof of Theorem~\ref{thm:zerodim}, for each variable $x_i$, a.a.s. $\randomgenset$ will contain a monomial of the form $x_i^j$ where $j\le r$.  If $j_1,\ldots,j_n$ are integers such that $1\le j_i\le r$ for each $1\le i\le n$, then the ideal $(x_1^{j_1},\ldots,x_n^{j_n})$ contains all monomials of degree $\geq rn$ in $S$, since if $\deg({x}^\alpha)\geq r$, by the pigeonhole principle at least one $\alpha_i\geq \lfloor\frac{r}{n}\rfloor$.  Hence, with probability tending to 1, $\randomideal$ will contain every monomial in $S$ of degree $nr$ and thus every minimal generator has degree at most $nr.$
\end{proof}

To illustrate this result, suppose that $p=\littleomega{{1}/{\log D}}$ and let $\randomideal\sim\umodel$.  Then, by Theorem~\ref{thm:degreecomplexity}, as $D$ tends to infinity, one should expect $\dc{\randomideal}$ to be at most $n\log D$. One should keep in mind that $\randomideal$ never contains generators of degree larger than $D$, and so the bound given in this theorem is not optimal for certain choices of $r$.

\section{Other probabilistic models}
\label{sec:othermodels}

As mentioned in the Introduction, the study of `typical' ideals from a family of interest may require not only the \modelname{}, but also 
more general models of random monomial ideals. To that end, we define the most general probabilistic model on sets of monomials 
which, a priori, provides no structure, but it does provide a framework within which other models can be recovered.
\paragraph{A general model.}     
Fix a degree bound $D>0$. To \emph{each} monomial $x^\alpha\in\kring$ with $0<\deg(x^{\alpha})\leq D$, the \emph{general model} for random monomials assigns an arbitrary probability $0\leq p_\alpha\leq 1$ of selecting the monomial $x^{\alpha}$:  
\begin{equation}
	\label{eq:generalModel}
	\prob{x^\alpha} = p_\alpha.
\end{equation} Hence the general model has  many parameters, namely $\{p_\alpha:\alpha\in\mathbb N^n\setminus\{0\},|\alpha|\leq D\}$.  It is clear that
that, for a fixed $n$ and fixed degree bound $D$, the  \modelname{} is  a  special case of the general model, where $p_\alpha=p(n,D)$, for all $\alpha\in\mathbb N^n$ such that $0<|\alpha|\le D$, does not depend on $\alpha$.  

Of course, there are many other interesting models one can consider. We define another  natural extension of \modelname{}. 
\paragraph{The graded model.}  
Fix a degree bound $D>0$. The \emph{graded model} for random monomials places the following probability on each monomial $x^\alpha\in\kring$ with $0<\deg(x^\alpha)\leq D$:
	\begin{equation}
 		p_\alpha=f(n,D,|\alpha|), 
	\end{equation}
where $0\leq p_\alpha\le 1$.  
Note that,  given $\deg(x^\alpha)=|\alpha|$, the probability of each monomial is  otherwise constant in $\alpha$. 
The graded model is a $D$-parameter 
 family of probability distributions on  random sets of monomials that induces a distribution on monomial ideals in the same natural way as the \modelname{}. Namely, the probability model selects random monomials and they are added to a generating set with probability according to the model, thus producing a random monomial ideal in $\kring$ whose generators are of degree at most $D$. 
More precisely, let $\mathbf{p}(n,D)=(p_1(n,D),\ldots,p_D(n,D))$ be such that $0\leq p_i(n,D)<1$ for all $1\le i\le D$.
 Initialize $B=\emptyset$ and for each monomial $x^\alpha$ of degree $i$, $1\leq i\leq D$, add $x^\alpha$ to $B$ with probability $p_i(n,D)$. 
 The resulting random monomial ideal is $I=\ideal{B}$ (again with the convention that if $B=\emptyset$, then we set $I=\ideal{0}$). 
Denote by $\gmodel$  the resulting induced distribution on monomial ideals.  As before, since ideals are now random variables, we will write  
$\randomideal\sim\gmodel$ for a random monomial ideal generated using the graded model.

The analogue of Theorem~\ref{thm:yuge} for $\gmodel$, which has an almost identical proof, is as follows:
\begin{thm}  \label{thm:hilbertDistributionGraded}
	Fix $n$, $D$, and the graded model parameters $\mathbf{p}(n,D)=(p_1,\dots,p_D)$. 
	For any fixed monomial ideal $\fixideal\subseteq S$, random monomial ideals from the graded model distribution $\randomideal\sim\gmodel$ satisfy the following: 
	\[ 
		P(\randomideal=\fixideal)=\prod_{d=1}^Dp_d^{\beta_{1,d}(S/\fixideal)}(1-p_d)^{h_\fixideal(d)}, 
	\]
	where $\beta_{1,d}(S/\fixideal)$ is the number of degree-$d$ minimal generators of $\fixideal$ (that is, the first graded Betti number of $S/I$), and $h_I(d)$ is its Hilbert function. 
\label{thm:gradedhilbdist}
\end{thm}

\subsection{Random monomial ideals generalize random simplicial complexes}

An \emph{abstract simplicial complex} on the vertex set $[n]$ is a collection of subsets (called \emph{faces}) of $[n]$ closed under the operation of taking subsets.   The \emph{dimension} $\dim(\sigma)$ of a face $\sigma\subseteq [n]$ is $|\sigma|-1$.  If $Y$ is an abstract simplicial complex on $[n]$, then the \emph{dimension} $\dim(Y)$ of $Y$ is $\max\{\dim(\sigma):\sigma\text{ is a face of } Y\}$
and the \emph{Stanley-Reisner ideal} $I_{Y}$ in the polynomial ring $\kring$ is the square-free monomial ideal generated by the monomials corresponding to the non-faces $\tau$ of $Y$: $I_{Y}\coloneqq\langle \tau:\tau\subseteq[n], \tau\not\in Y\rangle.$ For example, if $n=2$ and $Y=\{\emptyset,\{1\},\{2\}\},$ then $I_{Y}=\langle x_1x_2\rangle\subset k[x_1,x_2].$  A fundamental result in combinatorial commutative algebra is that the Stanley-Reisner correspondence constitutes a bijection between abstract simplicial complexes on $[n]$ and the square-free monomial ideals in $\kring$. See, for example,  \cite[Theorem 1.7]{miller+sturmfels}. 

In  \cite{costafarber}, Costa and Farber  introduced a model for generating abstract simplicial complexes at random.  Their model, which we call
the \emph{Costa-Farber model}, is a probability distribution on $\Delta_n^{(r)}$, the set of all abstract simplicial complexes on $[n]$ of dimension at most $r$. In it, one selects a vector $\mathbf{\tilde{p}}=(\tilde{p}_0,\ldots,\tilde{p}_r,0,\ldots,0)$ of probabilities such that $0\le \tilde{p}_l\le1$ for all $0\le l\le r$.  Then, one retains each of the $n$ vertices with probability $\tilde{p}_0$, and for each pair $ij$ of remaining vertices, one forms the edge $ij$ with probability $\tilde{p}_1$, and for each triple of vertices $ijk$ that are pairwise joined by an edge, one forms the $2$ dimensional face $ijk$ with probability $\tilde{p}_2,$ and so on. 

One can study the combinatorial and topological properties of the random simplicial complexes generated by the Costa-Farber model.
In view of the Stanley-Reisner correspondence, this model can be seen as a model for generating random square-free monomial ideals.  
Thus, 
the general model for generating random monomial ideals can be viewed as a generalization of the Costa-Farber model, which in turn has been
proved to generalize many other models of random combinatorial objects, for example, the Erd\H os-R\'enyi model for random graphs and  Kahle's model for random clique complexes \cite{kahlesurvey}. See \cite[Section 2.3]{costafarber} for details.

The relationship between our models for random monomial ideals and the Costa-Farber model can be made precise in the following way: there exists a choice of parameters $p_{\alpha}$ in \eqref{eq:generalModel} such that the resulting distribution on square-free monomial ideals in $S=\kring$ is precisely the distribution on the abstract simplicial complexes on $[n]$ under the Costa-Farber model.  

\begin{thm} \label{thm:CostaFarberCorresp}
Let $\mathbf{\tilde{p}}=(\tilde{p}_0,\tilde{p}_1,\ldots,\tilde{p}_{r},0,\ldots,0)$ denote the $n$-vector of probabilities in the Costa-Farber model for random simplicial complexes. Let $Y\subset\Delta_n^{(r)}$ be a simplicial complex on $[n]$ of dimension at most $r$ and let $I_Y$ be the Stanley-Reisner ideal corresponding to $Y$.  
Fix $D = r + 1$ and specify the following probabilities $p_{\alpha}$
, where $\alpha\in\mathbb{N}^n$ and $0<||\alpha||_1\le r+1$, for the general monomial model~\eqref{eq:generalModel}: 
\begin{equation}
\label{eq:palphas}
p_{\alpha}=
\begin{cases}
1-\tilde{p}_{\deg(x^{\alpha})-1},&\mbox{if } 0\not=\alpha\in\{0,1\}^n,\\
0,&\mbox{otherwise }.\\
\end{cases}
\end{equation}
Then, $P_{CF}(Y)=P(I_Y)$, where the former is probability under the Costa-Farber model and the latter under the distribution on random monomial ideals induced by the general model~\eqref{eq:generalModel}.  
\label{thm:CF}
\end{thm}
In other words, the specification of probabilities in  Theorem~\ref{thm:CostaFarberCorresp} recovers the Costa-Farber model on random simplicial complexes as a sub-model of the model~\eqref{eq:generalModel} on random monomial ideals.   
Note that this specification of probabilities can be considered as an instance of the graded model with support restricted to the vertices of the unit hypercube.

\begin{proof}  From \cite[Equation (1)]{costafarber}, the following probability holds under the Costa-Farber model:
\[
P_{CF}(Y)=\prod_{i=0}^r \tilde{p}_i^{f_i(Y)}(1-\tilde{p}_i)^{e_i(Y)},
\]
where $f_i(Y)$ denotes the number of $i$-dimensional faces of $Y$ and $e_i(Y)$ denotes the number of $i$-dimensional \emph{minimal} non-faces of $Y$ (i.e., the number of $i$-dimensional non-faces of $Y$ that do not strictly contain another non-face).  The minimal non-faces of $Y$ correspond exactly to the minimal generators of the Stanley-Reisner ideal $I_Y$.  Thus, $I_Y$ has exactly $e_i(Y)$ minimal generators of degree $i+1$, that is, $e_i(Y)=\beta_{1,i+1}(S/I_Y)$.  Each $i$-dimensional face of $Y$ corresponds to a degree $i+1$ standard square-free monomial of $I_Y$, hence  we have a Hilbert function value$h_{I_Y}(i+1)=f_i(Y).$  Next, note that the specification of probabilities in~\eqref{eq:palphas} depends only on the degree of each monomial, so that if $\deg(x^{\alpha})=\deg(x^{\alpha'})$, then $p_{\alpha}=p_{\alpha'}.$  Hence, for each $1\le j\le r+1,$ we denote by $p_j$ the probability assigned to the degree $j$ monomials in~\eqref{eq:palphas}, so $p_j=1-\tilde{p}_{j-1}$.  We now apply Theorem~\ref{thm:gradedhilbdist} to conclude that
$$P(I_Y)=\prod_{j=1}^{r+1}p_j^{\beta_{1,j}}(1-p_j)^{h_{I_Y}(j)}=\prod_{i=0}^{r}(1-\tilde{p}_i)^{e_i(Y)}\tilde{p}_i^{f_i(Y)}=P_{CF}(Y), \ \text{as desired.}$$  \end{proof}

{\small
\begin{table}[h]
	\begin{center}
	\begin{tabularx}{1.0\textwidth}{cccccc}
\toprule

		Simplicial \\complex $Y$ & 
		Ideal $I_Y$ &
		Faces of $Y$ &
		Non-faces of $Y$ &
		$P_{\mathrm{CF}}(Y)$ &
		$P(I_Y)$ \\
\midrule
			\rule{0pt}{0ex}
			Void
			& $k[x_1,x_2]$
			&none
			&$\emptyset,\ x_1,\ x_2,\ x_1x_2$
			&$0$
			&$0$\\
			\rule{0pt}{4ex}
			$\emptyset$
			& $\ideal{x_1,x_2}$
			& $\emptyset$
			&$x_1,\ x_2,\ x_1x_2$
			&$(1-\tilde{p}_0)^2$
			&$p_{1}^2$\\
			\rule{0pt}{4ex}
			\begin{tikzpicture}[baseline=-8]
    			\tikzstyle{point}=[circle,thick,draw=black,fill=black,inner sep=0pt,minimum width=4pt,minimum 				height=4pt]
    			\node (a)[point, label={[label distance=0mm]270:$1$}] at (0,0) {}; 
			\end{tikzpicture}
			&$\ideal{x_2}$
			&$\emptyset,\ x_1$
			&$x_2,\ x_1x_2$
			&$\tilde{p}_0(1-\tilde{p}_0)$
			&$p_{1}(1-p_{1})$\\
			\rule{0pt}{4ex}
			\begin{tikzpicture}[baseline=-8]
    			\tikzstyle{point}=[circle,thick,draw=black,fill=black,inner sep=0pt,minimum width=4pt,minimum 				height=4pt]
    			\node (a)[point, label={[label distance=0mm]270:$2$}] at (0,0) {}; 
			\end{tikzpicture}
			&$\ideal{x_1}$
			&$\emptyset,\ x_2$
			&$x_1,\ x_1x_2$
			&$\tilde{p}_0(1-\tilde{p}_0)$
			&$p_{1}(1-p_{1})$\\
			\rule{0pt}{4ex}
			\begin{tikzpicture}[baseline=-8]
    			\tikzstyle{point}=[circle,thick,draw=black,fill=black,inner sep=0pt,minimum width=4pt,minimum 				height=4pt]
    			\node (a)[point, label={[label distance=0mm]270:$1$}] at (0,0) {};
    			\node (b)[point, label={[label distance=0mm]270:$2$}] at (1.0,0) {};    
			\end{tikzpicture}
			&$\ideal{x_1x_2}$
			&$\emptyset,\ x_1,\ x_2$
			&$x_1x_2$
			&$\tilde{p}_0^2(1-\tilde{p}_1)$
			&$(1-p_{1})(1-p_{1})p_{2}$\\
			\rule{0pt}{4ex}
			\begin{tikzpicture}[baseline=-8]
    			\tikzstyle{point}=[circle,thick,draw=black,fill=black,inner sep=0pt,minimum width=4pt,minimum 				height=4pt]
    			\node (a)[point, label={[label distance=0mm]270:$1$}] at (0,0) {};
    			\node (b)[point, label={[label distance=0mm]270:$2$}] at (1.0,0) {};    
    			\draw (a.center) -- (b.center);
			\end{tikzpicture}
			&$\ideal{0}$
			&$\emptyset,\ x_1,\ x_2,\ x_1x_2$
			&none
			&$\tilde{p}_0^2\tilde{p}_1$
			&$(1-p_{1})(1-p_{1})(1-p_{2})$\\
\bottomrule
	\end{tabularx}
	\vspace{1mm}
	\label{tab:CFcorresp}
	\caption{A table illustrating the correspondence between random monomial ideals and the Costa-Farber model.}
\end{center}
\end{table}
}

Table~\ref{tab:CFcorresp} illustrates Theorem~\ref{thm:CF} in the case $n=2$ and $r=1$.  The Costa-Farber model generates a random simplicial complex on $\{1,2\}$ by specifying a probability vector $\mathbf{\tilde{p}}=(\tilde{p}_0,\tilde{p}_1)$ and starting with two vertices $x_1$ and $x_2$.  We retain each vertex independently with probability $\tilde{p}_0$.  In the event that we retain both $x_1$ and $x_2$, we connect them with probability $\tilde{p}_1$. By Theorem~\ref{thm:CF}, to recover this model from Equation~\eqref{eq:generalModel} we  set $D=2$ and $p_{(2,0)}=p_{(0,2)}=0$, $ p_1=p_{(1,0)}=p_{(0,1)}=1-\tilde{p}_0$, and $p_2=p_{(1,1)}=1-\tilde{p}_1.$  The probability distribution of the set of all simplicial complexes $Y$ on $\{1,2\}$ and, equivalently, all square-free monomial ideals $I_Y\subseteq k[x_1,x_2]$ appear in the table.

To conclude, note that there are other models that do not sample individual monomials, but instead sample ensembles of monomials all at the same time.  Examples of these are the methods to generate random integer partitions and Ferrers diagrams \cite{pittel} and random lattice polytopes \cite{baranymatousek}.

\section{Experiments \& conjectures}   \label{sec:experiments}

This section collects experimental work with two purposes in mind: illustrating the theorems we proved in this paper, and also motivating further conjectures about the \modelname{}. These experiments were performed using {\tt Macaulay2} \cite{M2} and  {\tt SageMath} \cite{sagemath}, including the \texttt{MPFR} \cite{mpfr} and \texttt{NumPy} \cite{numpy} modules. Graphics were created using {\tt SageMath}. 

We ran several sets of experiments using the  \modelname{}.  All experiments considered varying number of variables $n$,  varying maximum degree $D$, and varying probability $p$. For each choice of a triple $(n,D,p)$, we generated a sample of $N=1000$ monomial ideals. We then computed some
algebraic properties and tabulated the results. The figures in this section summarize the results and support the conjectures we state.

\paragraph{Cohen-Macaulayness.} 
Recall that an ideal has an (arithmetically) \emph{Cohen-Macaulay} quotient ring if its depth equals its Krull dimension. 
As is well-known, Cohen-Macaulayness is a very special property from which one derives various special results; see
\cite{BrunsHerzog} for a standard reference.  

\begin{figure}[h]	
	\includegraphics[width=\textwidth]{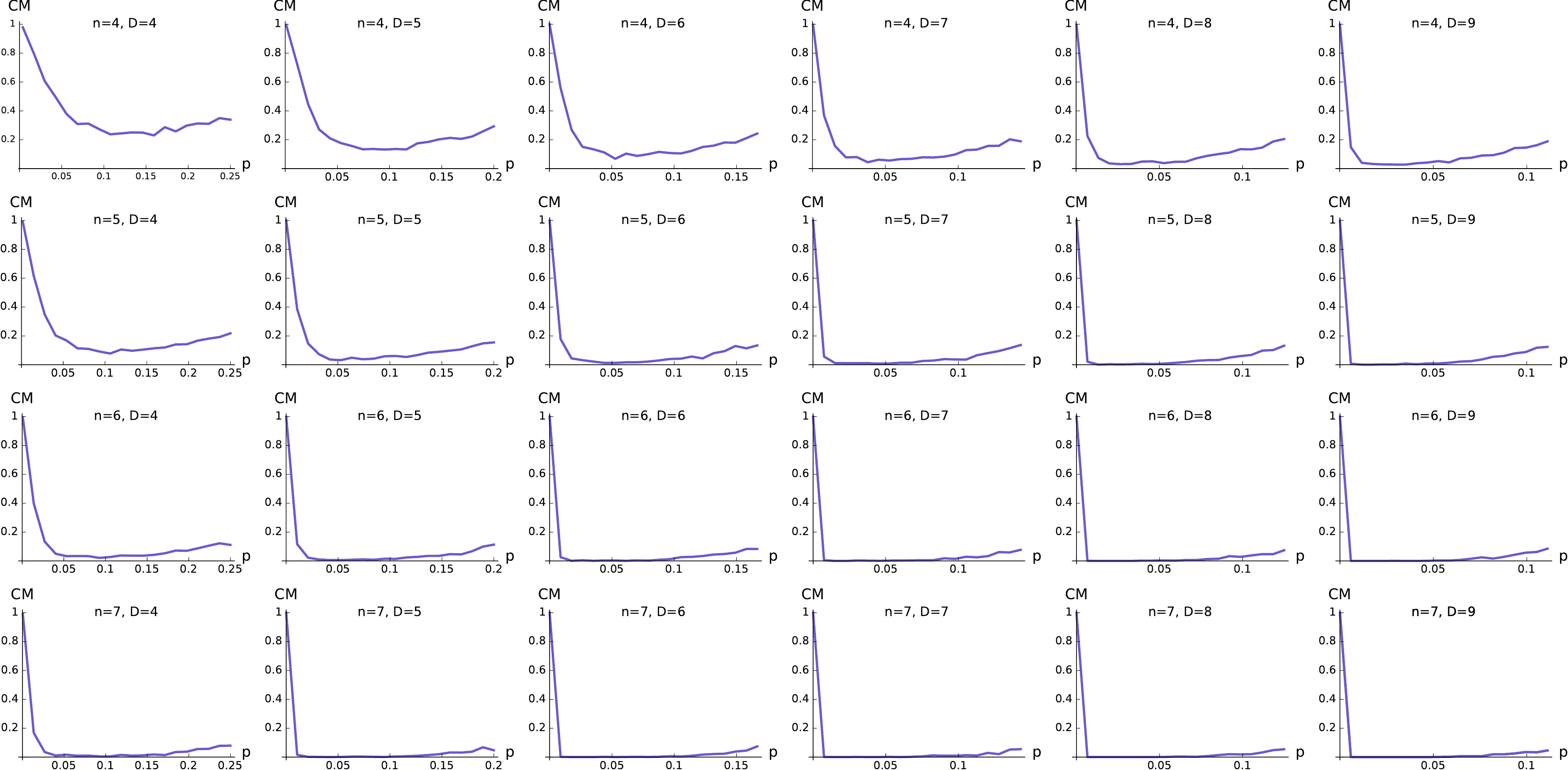}
	\caption{Frequency of random monomial ideals  $\randomideal\sim\umodel$ for which $S/\randomideal$ is  Cohen-Macaulay  for $(n,D)=(4,4)$ through $(n,D)=(7,9)$, as the model parameter $p$ takes values between $D^{-n}$ and ${D}^{-1}$. The sample size for each fixed value of $p$ is $1000$.}   
	\label{fig:CM:grid}
\end{figure}

Figure \ref{fig:CM:grid} shows the percent of ideals in the random samples whose quotient rings  are  arithmetically Cohen-Macaulay for $n=4,\dots,7$, $D=4,\dots,9$.  We see that, as $n$ and $D$ get larger, very few Cohen-Macaulay ideals are generated, suggesting that Cohen-Macaulayness is a ``rare" property in a meaningful sense. We speculate that the appearances of Cohen-Macaulay ideals are largely or entirely due to the appearance of the zero ideal and of zero-dimensional ideals, both of which are trivially Cohen-Macaulay. For the smallest $n$ and $D$ values, the zero ideal appears with observable frequency for all $p$ values, as do zero-dimensional ideals. As $p$ decreases toward the zero ideal threshold or increases toward the zero-dimensional threshold, these probabilities grow, resulting in a ``U" shape. (Note that probabilities continue to increase to 1 for larger $p$ beyond the domain of these plots.)  However, as $n$ and $D$ grow, the thresholds for the zero ideal and for zero-dimensionality become pronounced, and even as $p$ approaches either threshold the frequency of Cohen-Macaulayness stays low. When $p$ is so low that only the zero ideal is generated, Cohen-Macaulayness is of course observed with probability 1, but as soon as nontrivial ideals are generated, the frequency plummets to near zero. These experimental results suggest the following conjecture: 

\begin{conj} \label{conj:CM} 
	For large $n$ and $D$, the only Cohen-Macaulay ideals generated by the \modelname{} are the trivial cases (the zero ideal or zero-dimensional ideals), with probability approaching $1$. In particular, using the threshold functions of Theorems \ref{thm:KrullDimThreshold} and \ref{thm:gradedBettiThreshold}, we conjecture that as $n$ goes to infinity, the probability that $\randomideal\sim\umodel$ is Cohen-Macaulay will go to zero for $p=p(D)$ satisfying $p(D)=\littleoh{D^{-1}}$ and $p(D)=\littleomega{D^{-n}}$.
\end{conj}


\paragraph{Projective Dimension.}

Further exploring the complexity of minimal free resolutions, and  how much the ranges of Betti numbers vary, we investigate 
the \emph{projective dimension} of $S/\randomideal$; i.e., the length of the minimal free resolution of $S/\randomideal$.

\begin{figure}[h]	
	\centering
	\includegraphics[width=\textwidth]{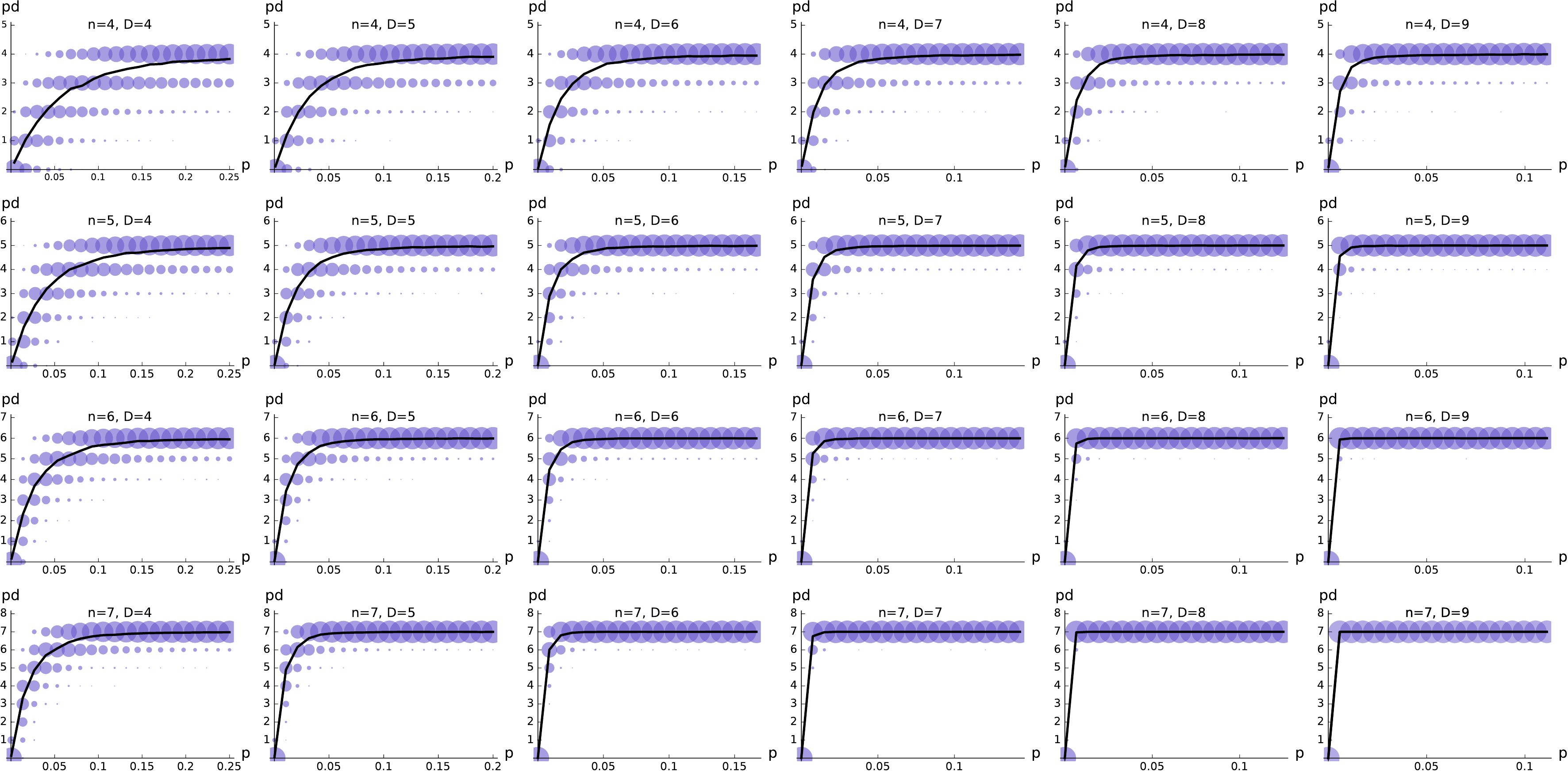}
	\caption{Projective dimension of $S/\randomideal$ for random monomial ideals $\randomideal\sim\umodel$ for $(n,D)=(4,4)$ through $(n,D)=(7,9)$. The values of the model parameter $p$ are 20 evenly spaced values between $D^{-n}$ and ${D}^{-1}$, and the sample size for each value of $p$ is $1000$. For each $p$, circle sizes indicate the proportion of quotient rings with that projective dimension, while the black curve indicates the mean projective dimension.}  
	\label{fig:projdim:grid}
\end{figure}

By the Hilbert syzygy theorem, the projective dimension is at most $n$. We see experimentally that for large $n$ and $D$, the projective dimension of $S$ modulo any non-zero random ideal tends toward this upper bound. (Since $\operatorname{pd}(S/\randomideal)=0$ if and only if $\randomideal=\ideal{0}$, the projective dimension will always concentrate at $0$ below the $p=D^{-n}$ threshold for the zero ideal.)

\begin{conj}\label{conj:pdim}
	As $n$ and $D$ increase, the projective dimension of the quotient ring of any non-zero ideal in the \modelname{} is equal to $n$ with probability approaching 1.
\end{conj}

Note that Conjecture \ref{conj:pdim} implies Conjecture \ref{conj:CM}. By the Auslander-Buchsbaum formula (\cite{BrunsHerzog}, see also \cite{stanley}), $\operatorname{pd}(S/\randomideal)=n-\operatorname{depth}(S/\randomideal)$ for all $\randomideal$. Hence Conjecture \ref{conj:pdim} implies $\operatorname{depth}(S/\randomideal)\to 0$, which implies that the quotient ring of a nonzero ideal will be Cohen-Macaulay if and only if it is zero-dimensional.


\paragraph{Strong genericity.}
A monomial ideal $I\subset \kring$ is said to be \emph{strongly generic} if no two minimal generators agree on a non-zero exponent 
of the same variable. For example, $\ideal{x_1x_2,x_1^2x_3}$ is a strongly generic monomial ideal in $k[x_1,x_2,x_3]$, but 
$\ideal{x_1x_2,x_1x_3}$ is not. Strong genericity is interesting because the minimal free resolution of $S/\randomideal$ has a combinatorial 
interpretation using the  \emph{Scarf polyhedral complex} (see \cite[Chapter 6]{miller+sturmfels} for details.), when  $\randomideal$ is strongly generic. 

\begin{figure}[h]	
	\includegraphics[width=\textwidth]{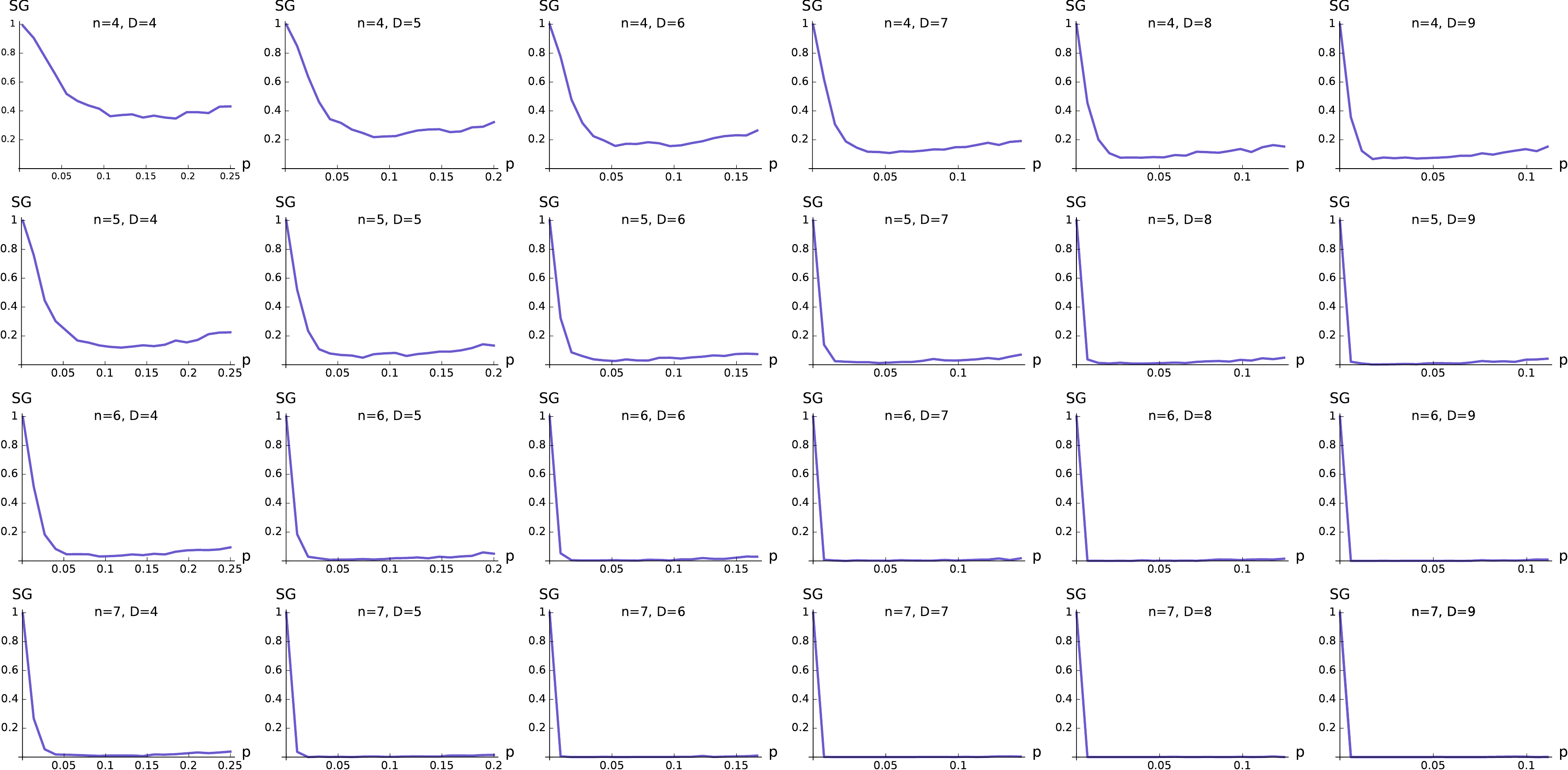}
	\caption{Frequency of strong genericity in the \modelname{}, for $(n,D)=(4,4)$ through $(n,D)=(7,9)$, as the model parameter $p$ takes values between $D^{-n}$ and ${D}^{-1}$. The sample size for each fixed value of $p$ is $1000$.}  
	\label{fig:gen:grid}
\end{figure}

The zero ideal, with no generators, is trivially strongly generic, as is any principal ideal. At first it might seem that increasing $p$, and thus increasing the expected number of monomials included in $\randomideal$, would always lower the frequency of strong genericity. However at the other extreme of $p=1-\epsilon$, the maximal ideal $\ideal{x_1,x_2,\ldots,x_n}$ occurs frequently, and this ideal is strongly generic. Between these extremal cases, we observe experimentally that very few random monomial ideals are strongly generic, and suspect a connection between this behavior and our results in Section \ref{sec:bettistuff} about the number of minimal generators.

\begin{conj} \label{conj:generic} 
	As $n$ and $D$ increase, there will be a lower threshold function $l(n,D)$ and an upper threshold function $u(n,D)$, such that the probability of being strongly generic will go to zero for $p=\littleoh{u(n,D)}$ and $p=\littleomega{l(n,D)}$, while the probability of being strongly generic will go to one for $p=\littleoh{l(n,D)}$ as well as for $p=\littleomega{u(n,D)}$.	
\end{conj}

Note that because the plots in Figure \ref{fig:gen:grid} display $p$ values between $D^{-n}$ and ${D}^{-1}$, the behavior as $p$ approaches 1 is not visible. 


We now turn to several properties of monomial ideals for which we experimentally observed noticeable patterns, but for which we do not have explicit conjectures to present. 

\paragraph{Castelnuovo-Mumford regularity.} 
The \emph{Castelnuovo-Mumford regularity} (or simply regularity) of an ideal is a measure of the variation and spread of the the degrees of the generators at each of the modules in a  minimal free resolution of the ideal. To compute the regularity requires computing the resolutions and the number of rows in the {\tt Macaulay2}-formatted Betti diagram. Of course, the degree complexity studied in  Section~\ref{sec:bettistuff}  is  similar but less refined measure of complexity, as it counts the number of entries in the first column of the Betti diagram only. 

\begin{figure}[h]	
	\includegraphics[width=\textwidth]{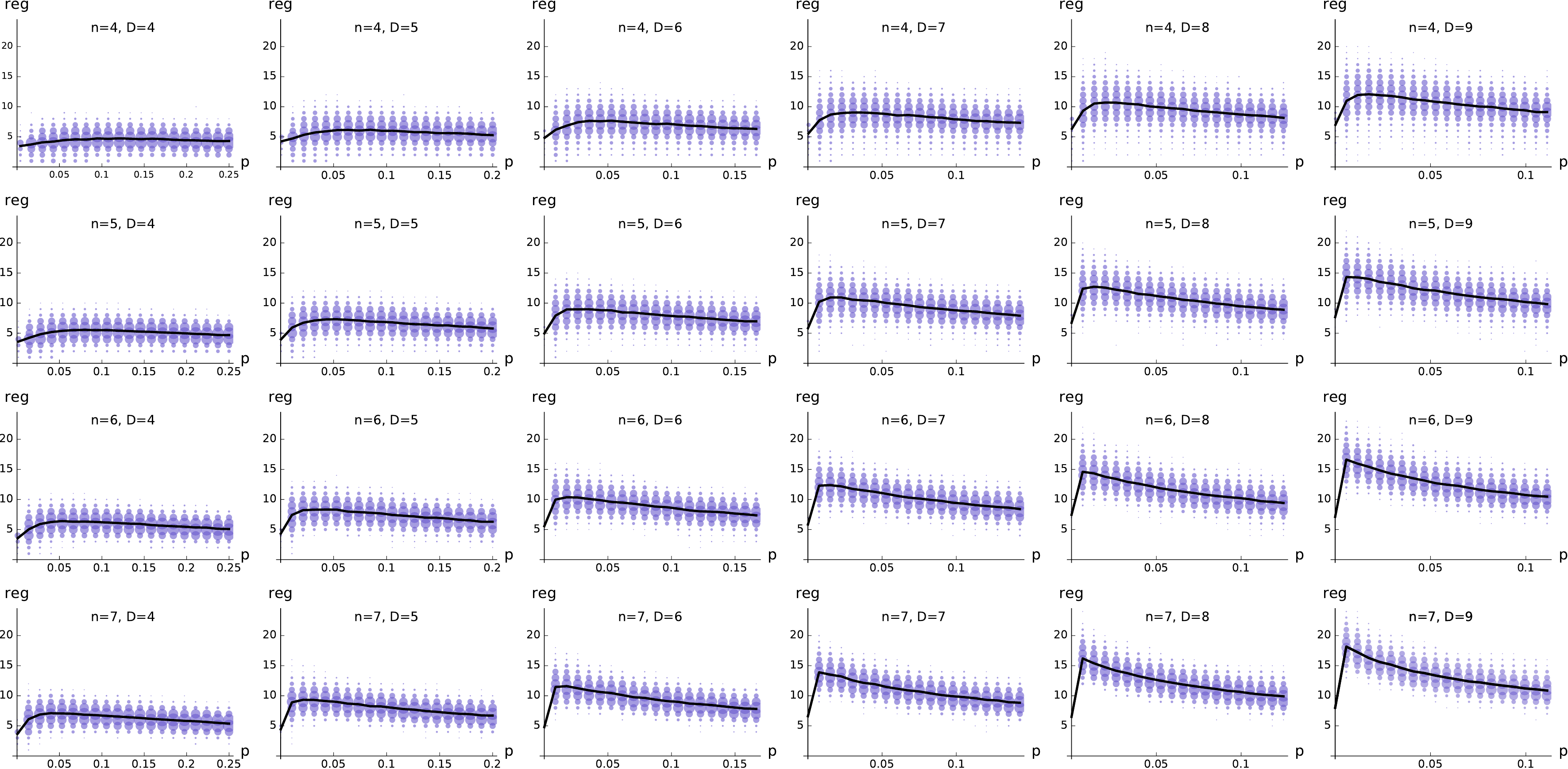}
	\caption{Castelnuovo-Mumford regularity of random monomial ideals \mbox{$\randomideal\sim\umodel$} for $(n,D)=(4,4)$ through $(n,D)=(7,9)$. Note that $n$ increases across rows and $D$ increases down columns, and that the values of the model parameter $p$ are 20 evenly spaced values between $D^{-n}$ and ${D}^{-1}$. The sample size for each fixed value of $p$ is $1000$. Circle sizes indicate the relative frequencies of occurrence of random monomial ideals of corresponding regularity, and the mean regularity for each $p$ value is indicated by the black curves.}  
	\label{fig:regularity:grid}
\end{figure}

The simulations on the range of values realized for regularity of random monomial ideals, which  do not take into account  any zero ideals generated 
(since their regularity is $-\infty$), show an interesting trend. Namely, already for the small values of $n$ and $D$ depicted in Figure~\ref{fig:regularity:grid}, there is quite a range of values obtained under the model. As $n$ grows these values concentrate more tightly around their mean, while seeming to follow a binomial distribution (discretized version of a normal distribution).

\paragraph{Simplicial homology of associated simplicial complexes.}

\begin{figure}[h]
	\centering
	\includegraphics[width=\textwidth]{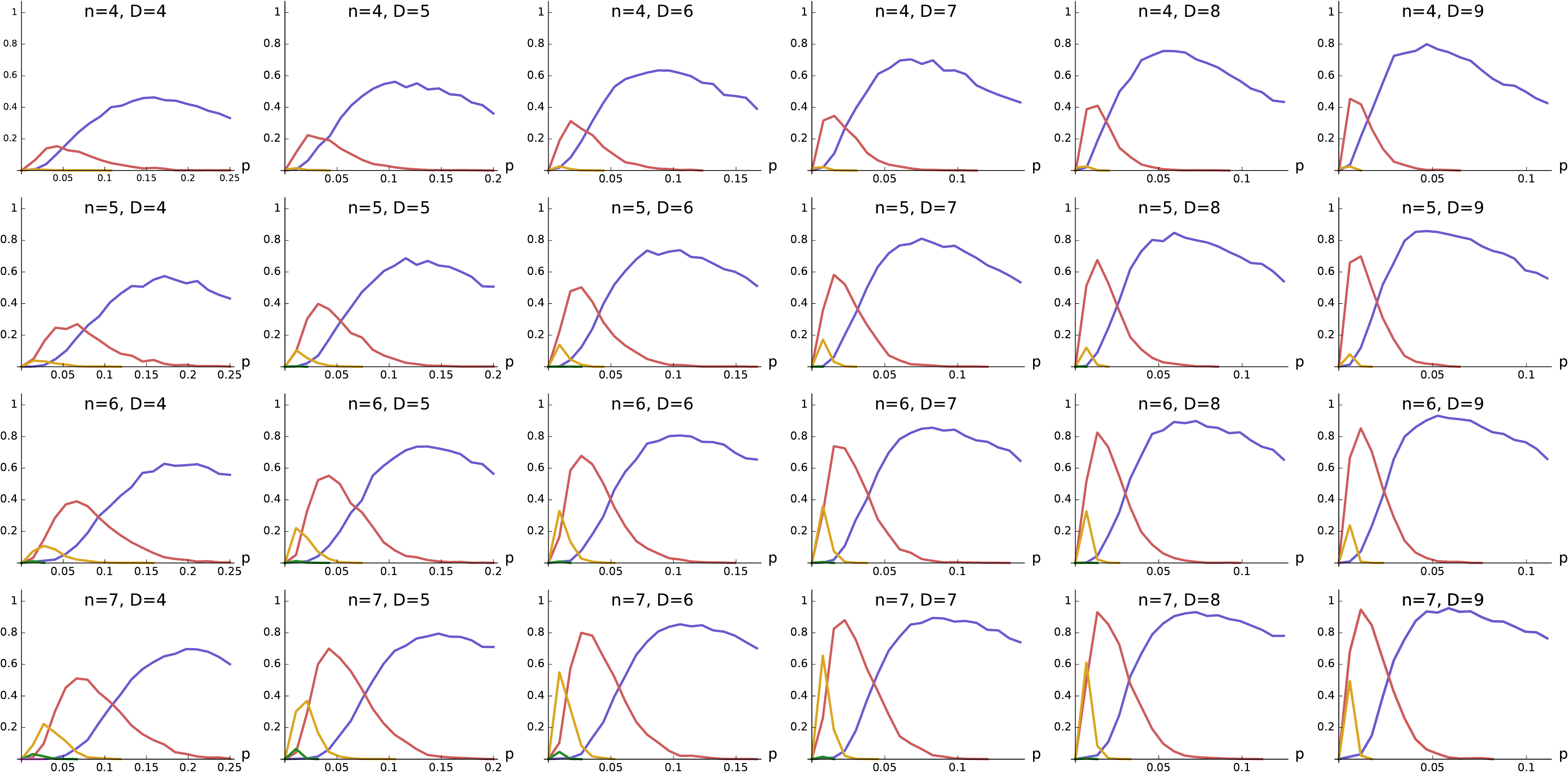}	
	\vspace*{.5em}
	\includegraphics[width=.7\textwidth]{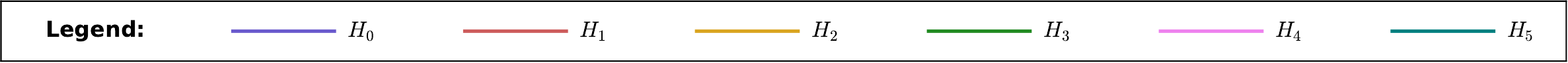}
	\caption{Non-trivial homology groups for the random simplicial complexes whose nonfaces are radicals of random monomial ideals in the \modelname{} setting. Each subplot in Figure \ref{fig:hom:radical} represents a fixed pair of $(n,D)$ values, with $n$ 
increasing down columns, and $D$ increasing across rows. For each $n,D$ and $p$ value we generated $1000$ ideals $\operatorname{rad}(\randomideal\sim\umodel)$. Each subplot registers the frequency at which the second kind of random simplicial complex  exhibited $H_i\neq \{0\}$ over ${\mathbb Z}_2$. 
The values of $p$ in each subplot are chosen in the interval $[D^{-n},D^{-1}]$. We note that $\randomideal$ tends to the maximal ideal as $p\to 1$, so every curve will tend to $0$.
}  \label{fig:hom:radical}
\end{figure}

\begin{figure}[h]
	\centering
	\includegraphics[width=\textwidth]{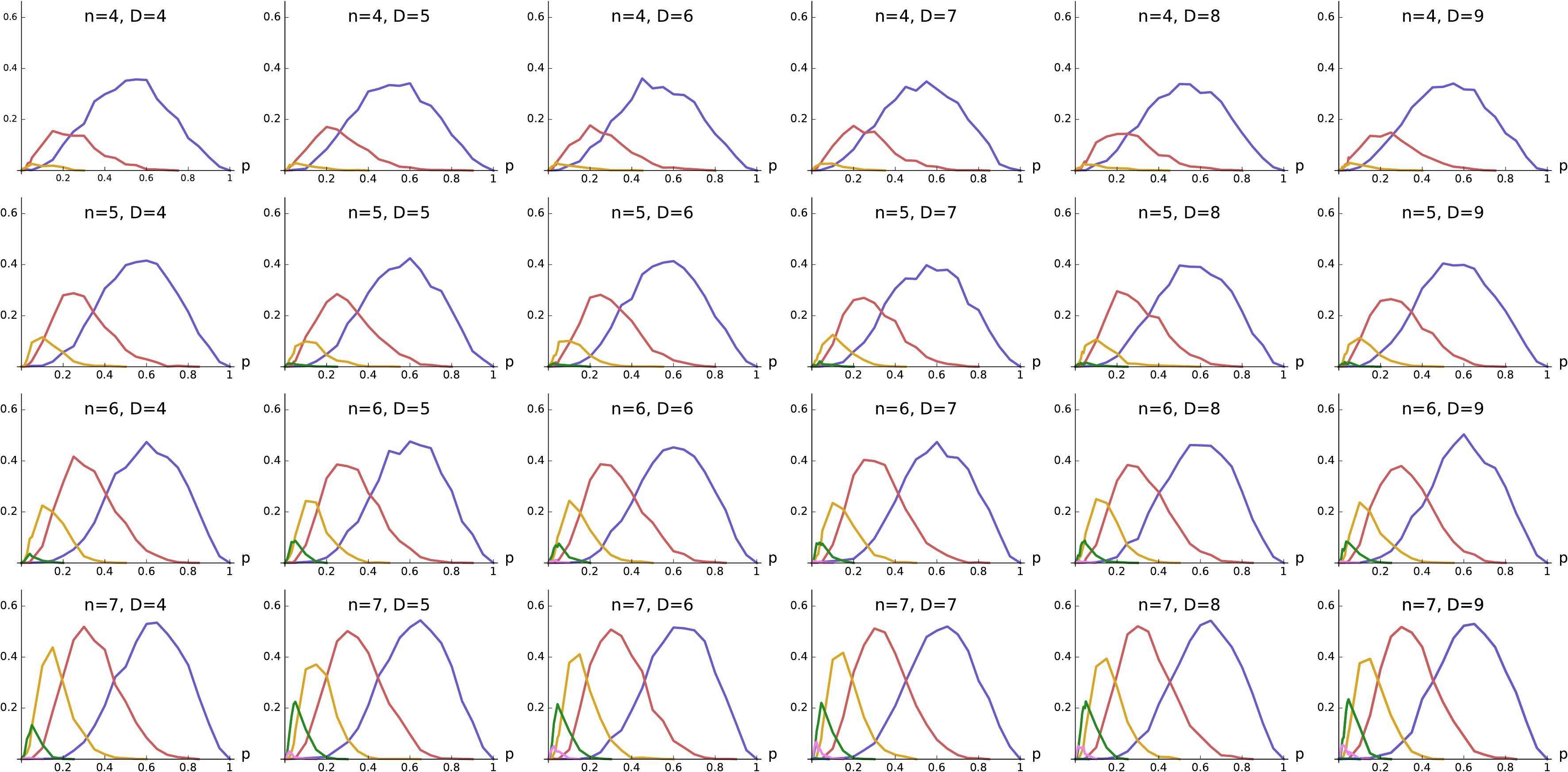}	
	\vspace*{.5em}
	\includegraphics[width=.7\textwidth]{hom_legend}
	\caption{Non-trivial homology groups for the random simplicial complexes whose non-faces are random square-free monomial ideals.
	Each subplot is for a pair $(n,D)$ values, with $n$ increasing down columns, and $D$ increasing across rows. With a sample size of 1000 monomial ideals for each $n,D$ and $p$, each subplot shows the frequency at which the second kind of random simplicial complex  exhibited $H_i\neq \{0\}$ over 
	${\mathbb Z}_2$. The $p$ values for this set of experiments are $\{0.05n\mid 0\leq n\leq 20\}$. The subplots register the frequency at which these randomly generated simplicial complexes exhibited $H_i\neq \{0\}$. 
	 }
	\label{fig:hom:squarefree}
\end{figure}

One motivation for studying random monomial ideals is that they provide a way to generate \emph{random simplicial complexes} different from earlier work, where the authors randomly generate $k$-dimensional faces of complexes for some fixed integer $k$ (see \cite{costafarber,kahlesurvey,linial+meshulam}). Instead, by Stanley-Reisner duality, the indices in a monomial we generate are the elements of a \emph{non-face} of a simplicial complex. 

There are two natural ways to randomly generate sets of square-free monomials and our experiments considered both of them.
First, as radicals of random monomial ideals drawn from the \modelname{} and, second, directly as random square-free monomial ideals drawn from the general model, as described in Theorem \ref{thm:CostaFarberCorresp}, that places zero probability on non-square-free-monomials and probability $p$ on square-free monomials. 
In our experiments, given an ideal in $n$ variables, we computed the ${\mathbb Z}_2$-homology $H_i$, for $i$ from 0 to $n-1$, of the associated simplicial complex. Figure \ref{fig:hom:radical} displays the homological properties of the random simplicial complexes obtained via the Stanley-Reisner 
correspondence from the first approach. Figure \ref{fig:hom:squarefree} displays the homological properties of the random simplicial complexes obtained via the second method, directly generating square-free ideals, avoiding taking the radical. 

The patterns of appearance and disappearance of higher homology groups, visible  in both figures, are familiar in the world of random topology; see, for example, \cite{kahle+meckes}. For each $i$ there is a lower threshold below which $H_i$ always vanishes, and an upper threshold above which $H_i$ also vanishes, while between these two thresholds we see a somewhat normal-looking curve. Since our model parameter $p$ controls how many faces are \emph{removed} from a complex, higher dimensional objects are associated with lower values of $p$: a reversal of the typical behavior in prior random topological models.  Another interesting pattern in both sets of experiments (unfortunately not evident in the plots) is that there was an overall tendency for each random simplicial complex to have no more than one non-trivial homology group. That is, if for a particular $n,D$ and $p$ the experiments found that $H_0$ was nontrivial with frequency $x$, and $H_1$ was nontrivial with frequency $y$, one might expect to see random simplicial complexes with $H_0$ and $H_1$ \emph{simultaneously} nontrivial with frequency $xy$. This would be the case if these events were statistically uncorrelated. However we consistently observed frequencies much lower than this, i.e. a negative correlation between nonzero homology at $i$ and nonzero homology at $j$, for every $i\neq j$.


\section{Acknowledgements} We are grateful for the comments, references and suggestions from Eric Babson, Boris Bukh, Seth Sullivant, Agnes Szanto, Ezra Miller, and Sayan Mukherjee. We are also grateful to REU student Arina Ushakova who did some initial  experiments for this project.
\bibliographystyle {acm} 
\bibliography{randomMonomialIdeals}
\end{document}